\documentclass[oneside]{amsart}
\usepackage{a4wide}
\usepackage{amssymb,amsmath}
\usepackage{enumerate}
\usepackage{comment}
\usepackage[shortlabels]{enumitem}
\usepackage{marginnote}
\usepackage{bbold}
\usepackage{smartref}
\usepackage{subcaption,nicefrac}
\usepackage[unicode,naturalnames,hidelinks]{hyperref}
\usepackage{bm}
\usepackage[utf8]{inputenc}
\usepackage[T1]{fontenc}
\usepackage{microtype}
\usepackage{mathscinet}
\usepackage{xcolor}

\DeclareMathOperator{\sign}{sgn}

\usepackage{accents}
\MakeRobust{\underaccent} 
\newcommand{\name}[1]{\underaccent{\sim}{#1}}
\newcommand{\neta}{{\name\eta}}
\newcommand{\ntau}{\name\tau}
\newcommand{\nsigma}{\name\sigma}
\newcommand{\na}{\name a}
\newcommand{\npi}{\name\pi}
\newcommand{\nphi}{\name \phi}

\DeclareMathOperator{\majority}{major}

\DeclareMathOperator{\poss}{poss}
\DeclareMathOperator{\POSSBACK}{POSS}
\newcommand{\QPOSS}{\POSSBACK^{Q}}
\newcommand{\PPOSS}{\POSSBACK}
\newtheorem{theorem}[equation]{Theorem}

\newtheorem{lemma}[equation]{Lemma}
\newtheorem{definitionandlemma}[equation]{Definition and Lemma}
\newtheorem{plemma}[equation]{Preliminary Lemma}

\newtheorem{corollary}[equation]{Corollary}

\theoremstyle{definition}
\newtheorem{definition}[equation]{Definition}
\theoremstyle{remark}
\newtheorem{remark}[equation]{Remark}
\newtheorem*{remark*}{Remark}
\newtheorem*{remarks*}{Remarks}
\newtheorem*{notation*}{Notation}
\newtheorem{assumption}[equation]{Assumption}
\newtheorem{fact}[equation]{Fact}
\newtheorem{facts}[equation]{Facts}
\newtheorem*{facts*}{Facts}

\numberwithin{equation}{section}

\DeclareMathOperator{\dom}{dom}

\DeclareMathOperator{\cf}{cf}
\DeclareMathOperator{\supp}{supp}

\newcommand{\odd}{\mathnormal{\textsc{odd}}}

\newcommand{\nx}{{\name x}}
\newcommand{\ny}{{\name y}}
\newcommand{\nX}{{\name X}}
\newcommand{\Gen}{f_{\text{gen}}}

\date{2024-03-04}

\thanks{Supported by 
Austrian Science Fund (FWF): grants P33420	and P33895
(first author) and T1081 (second author), and
Israel Science Foundation (ISF) grant 2320/23 (third author).
Variants of Sections~\ref{sec:ma} and~\ref{sec:meas} where 
included in the second author's thesis.
This is publication number~1224 of the third author.}

\author{Jakob Kellner}
\address{Technische Universität Wien (TU Wien).}
\email{jakob.kellner@tuwien.ac.at}
\urladdr{\url{http://dmg.tuwien.ac.at/kellner/}}
\author{Anda Ramona T\u anasie}
\address{Technische Universität Wien (TU Wien).}
\email{anda-ramona.latif@tuwien.ac.at}
\author{Saharon Shelah}
\address{The Hebrew University of Jerusalem and Rutgers University.}
\email{shlhetal@mat.huji.ac.il}
\urladdr{\url{http://shelah.logic.at/}}

\title{On automorphisms of $\mathcal P(\lambda)/[\lambda]^{<\lambda}$.}

\begin{document}
\begin{abstract}
We investigate the statement ``all automorphisms  of $\mathcal P(\lambda)/[\lambda]^{<\lambda}$ are trivial''. We show that MA implies the statement for regular uncountable $\lambda<2^{\aleph_0}$; 
that the statement is false for measurable $\lambda$ if $2^\lambda=\lambda^+$; and that for ``densely trivial'' it can be forced (together with $2^\lambda=\lambda^{++}$) for inaccessible~$\lambda$.
\end{abstract}

\maketitle

\section{Introduction}
\newcommand{\myte}{T1}
\newcommand{\mytz}{T2}
\newcommand{\mytd}{T3}
\newcommand{\mytv}{T4} 
We investigate automorphisms of Boolean algebras of the form
\[P^\lambda_\kappa:=\mathcal P(\lambda)/[\lambda]^{<\kappa}\]
The instance $P^{\omega}_{\omega}$, i.e., $\mathcal P(\omega)/\mathrm{FIN}$, has
been studied extensively for many years.\footnote{Rudin~\cite{rudin1956,MR80903} showed in the 1950s that CH implies that there is a non-trivial automorphism; Shelah~\cite{MR675955} showed that consistently all automorphisms are trivial. Further results can be found, e.g., in~\cite{Sh:987,Sh:990,MR1202874,MR1896046,MR3880881,MR1711328,farah2024corona}.} 
One can study variants for uncountable cardinals $\lambda$.
Unsurprisingly, the behaviour here 
tends to be quite different to the countable case.
One moderately popular\footnote{See e.g.~\cite{MR1202874,MR3549381,MR3846847}.} such generalisation is $P^\lambda_{\omega}$.
Here, we study another obvious generalization of the countable case, $P^\lambda_{\lambda}$.
Some results for general $P^\lambda_\kappa$ can be found in~\cite{MR3480121}.

The main result of the paper is: 
\begin{equation}
\tag{\myte, Thm.~\ref{thm:main4}} \parbox{0.8\textwidth}{The following is equiconsistent  with an inaccessible: $\lambda$ is inaccessible,
    $2^\lambda$ is $\lambda^{++}$ 
    and all automorphisms of $P^\lambda_{\lambda}$
    are densely trivial.}    
\end{equation}
Here, $2^\lambda>\lambda^{+}$ is necessary, at least for measurables:
\begin{equation}
\tag{\mytz, Thm.~\ref{thm:meas}} \parbox{0.8\textwidth}{If $\lambda$ is measurable and 
    $2^\lambda=\lambda^+$, then there is a nontrivial automorphism
    of $P^\lambda_{\lambda}$.}
\end{equation}
\begin{remark} 
From~\cite[Lem.~3.2]{Sh:990} it would follow that \mytz\ holds even when 
``measurable'' is replaced by just ``inaccessible''.
However, the proof there turned out to be incorrect.\footnote{A corrected version has been submitted, see~\url{https://shelah.logic.at/papers/990a/}. This version again establishes the result only assuming inaccessibility.}
\end{remark}

For $\lambda$  below the continuum
we get the following result under Martin's Axiom (MA). More explicitly, MA$_{=\lambda}(\sigma\text{-centered})$ is sufficient,
which is the statement that for any $\sigma$-centered poset $P$ and 
${\le}\lambda$ many open dense sets in $P$ there is a filter $G$ meeting all these open sets:
\begin{equation}\tag{\mytd, Thm.~\ref{Matriv}}
\parbox{0.8\textwidth}{For $\aleph_0<\kappa\le\lambda<2^{\aleph_0}$ and $\kappa$ regular, MA$_{=\lambda}(\sigma\text{-centered})$  implies that every automorphism of $P^\lambda_{\kappa}$ is trivial.}    
\end{equation}

Larson and McKenney~\cite{MR3480121} showed the same
under MA$_{\aleph_1}$ for the case  $\lambda=2^{\aleph_0}$ and $\kappa=\aleph_1$.

Contrast this to the case $\lambda=\kappa=\omega$:
Due to results of Veli\v{c}kovi\'{c}, Stepr\=ans and the third author,
``Every automorphisms
of $\mathcal P(\omega)/[\omega]^{<\omega}$ is trivial''
is implied by PFA~\cite{Sh:315}, in fact even by 
MA+OCA~\cite{MR1202874},
but not by MA alone~\cite{MR1202874}
(not even for ``somewhere trivial''~\cite{MR1896046}).




\subsection*{Contents}
    We start by introducing some notation and basic results
    in Sec.~\ref{sec:def} (p.~\pageref{sec:def}).
    
    The following sections are independent of each other: 
    
    In Sec.~\ref{sec:ma} (p.~\pageref{sec:ma}) we show 
    \mytd, i.e., Thm.~\ref{Matriv}; 
    in Sec.~\ref{sec:meas} (p.~\pageref{sec:meas}), 
    we show \mytz, i.e., Thm.~\ref{thm:meas}; and finally in the main part, Sec.~\ref{sec:forc}  (p.~\pageref{sec:forc})
    we develop some forcing notions to prove \myte, i.e., Thm.~\ref{thm:main4}.

\subsection*{Acknowledgments}
We thank an anonymous referee for numerous corrections.

\section{Definitions}\label{sec:def}
We always assume:
\begin{itemize}
    \item $\lambda$ is a cardinal and $\kappa\le\lambda$ is regular.
    \item The case $\kappa=\aleph_0$ or $\lambda=\aleph_0$
    is included only for completeness sake in the following definitions.
    \item In  Section~\ref{sec:ma} we will assume that $\aleph_1\le \kappa \le \lambda< 2^{\aleph_0}$.
    \item In Section~\ref{sec:meas} we assume that $\lambda$ is measurable and $\kappa=\lambda$.
    \item In Section~\ref{sec:forc} we assume that $\lambda$ is inaccessible and $\kappa=\lambda$.
\end{itemize}
Notation:
\begin{itemize}
    \item  We investigate 
    the 
Boolean algebra (BA)  $P^\lambda_\kappa:=\mathcal P(\lambda)/[\lambda]^{{<}\kappa}$,
i.e., the power set of $\lambda$ factored by the ideal of
sets of size ${<}\kappa$.
    \item For $A\subseteq \lambda$, we 
    denote the equivalence class of $A$ with  
    $[A]$. We set $\mathbb 0:=[\emptyset]$.
    \item $A\subseteq^* B$ means $|B\setminus A|<\kappa$,
    analogously for $A=^*B$; and ``for almost all $\alpha\in A$''
    means for all but ${<}\kappa$ many in $A$.
In particular, $A=^*\lambda$ means $A\subseteq \lambda$ and
$|\lambda\setminus A|<\kappa$.
\item We denote the BA-operations in 
$P^\lambda_\kappa$ with $x \vee y$, 
$x \wedge y $ and $x^c$ (for the complement).

So we have $[A]\vee [B]=[A\cup B]$,
$[A]\wedge [B]=[A\cap B]$, and $[A]^c=[\lambda\setminus A]$.
\item A function $\phi: P^\lambda_\kappa\to P^\lambda_\kappa$ is a BA-automorphism (which we will just
call \emph{automorphism}), if it is bijective,  
compatible with $\wedge$ and the complement, and satisfies
$\phi(\mathbb 0)=\mathbb 0$.
\item Preimages of a function $f$ are denoted by $f^{-1}x$, images
by $f''x$. 
\item We sometimes 
identify $\eta\in 2^\lambda$  with
$\eta^{-1}\{1\}\subseteq \lambda$
 without explicitly mentioning it, by 
referring to $\eta$ as element of $2^\lambda$
or of $P(\lambda)$.
\end{itemize}


Let us note that 
$P^\lambda_\kappa$ is ${<}\kappa$-complete\footnote{I.e.,
if  $|I|<\kappa$ then 
$\bigvee_{i\in I}[A_i]=[\bigcup_{i\in I} A_i]$}
and $\lambda^+$-cc.
Also, 
any automorphism $\phi$ is closed under
${<}\kappa$ unions:  
$\phi(\bigvee_{i\in I}[A_i])=
\bigvee_{i\in I}\phi([A_i])$.

An automorphism is trivial if it is induced by a function
on $\lambda$. A standard definition to capture this concept is the following:
\begin{definition}
  An automorphism $\phi:P^\lambda_\kappa\to P^\lambda_\kappa$
  is trivial, if there is a $g:\lambda\to\lambda$ 
  such that $\phi([A])=[g^{-1}A]$ for all $A\subseteq \lambda$.
\end{definition}

However, we prefer to use forward images instead of inverse images; which can easily be seen to be equivalent:
\begin{definition}\mbox{}
  \begin{itemize}
      \item  For $f:A_0 \to\lambda$ with $A_0=^* \lambda$, 
   define $\pi_f:P^\lambda_\kappa\to P^\lambda_\kappa$
   by $\pi_f([B]):=[f''(B\cap A_0)]$
   for all $B\subseteq \lambda$.
      \item $f$ is an almost permutation, if there are
      $A_0=^*\lambda$ and
      $B_0=^*\lambda$ with $f:A_0\to B_0$ bijective.
  \end{itemize}

\end{definition}
(Such a $\pi_f$ is always a well-defined function.)

\begin{lemma} Let $\phi:P^\lambda_\kappa\to P^\lambda_\kappa$ be a function. The following are equivalent:
\begin{enumerate}
    \item  $\phi$ is a trivial automorphism.
    \item  There is an almost permutation $f$ such that $\phi=\pi_f$.
    \item  (Assuming $\kappa>\aleph_0$.) There is a bijection $f:\lambda\to\lambda$ such that
    $\phi=\pi_f$.
\end{enumerate}
\end{lemma}

\begin{proof}
(1) implies (2):
Assume $\phi$ is a trivial automorphism, witnessed by $g$. 

Then
$X:=g''\lambda=^*\lambda$ (as $\phi([X])=[g^{-1}X]= [\lambda]$),
and 
$Y:=\{\alpha\in X:\, |g^{-1}\{\alpha\}|\ne 1\}=^*\emptyset$: Otherwise,
pick $y^0_\alpha\ne y^1_\alpha$ for each $\alpha\in Y$ with $g(y^0_\alpha)=g(y^1_\alpha)=\alpha$.
So $y^0_\alpha\in g^{-1}C$ iff
$y^1_\alpha \in g^{-1}C$ for any $C\subseteq \lambda$.
Set $B^i:=\{y^i_\alpha:\, \alpha\in Y\}$ for $i=0,1$ and let 
$[C]=\phi^{-1}([B^0])$.
So $\phi([C])=[g^{-1}C]=[B^0]$, i.e., almost all $y^0_\alpha$ are in
$g^{-1}C$, but then almost all $y^1_\alpha$ are in
$g^{-1}C$ as well, i.e., $[B^0]=\phi([C])\ge [B^1]$, a contradiction
as $B^0\cap B^1=\emptyset$.

Set $A_0:=X\setminus Y$, and $B_0:=g^{-1}A_0$. Note that $B_0=^*\lambda$, as
$\mathbb 0=\phi(\mathbb 0)=\phi([Y])=[g^{-1}Y]$. So $g\restriction B_0\to A_0$
is bijective, and we can set $f:A_0\to B_0$ the inverse. Then $f$ is an almost
permutation, and $\pi=\pi_f$.

(2) implies (1):
Let $f:A_0\to B_0$ be an almost permutation, and $g:B_0\to A_0$ the inverse
(and let $g$ be defined arbitrarily on $\lambda\setminus B_0$).
Then $\pi_f([X])=[f''(X\cap A_0)]=[g^{-1}(X)]$. It remains to be shown that
$\pi_f$ is an automorphism: $\pi_f([\emptyset])=[f''\emptyset]=[\emptyset]$;
$\pi_f([X\cap Y])= [f''(X\cap Y\cap A_0)]=[f''(X\cap A_0)\cap f''(Y\cap A_0)]$;
and $\pi_f([\lambda \setminus X])= [f''(A_0\setminus X)]=[B_0\setminus f''X]$.

(2) implies (3) if $\cf(\kappa)>\aleph_0$: This follows from the follwing lemma.
\end{proof}

\begin{lemma}\label{lem:almostisall} ($\kappa>\aleph_0$)
Let $f$ be a $\kappa$-almost permutation. Then there is an 
$S=^*\lambda$ such that  $f\restriction S:\, S\rightarrow S$ is bijective.
\end{lemma}
\begin{proof}
Set $X_0:=A_0=\dom(f)$, and $X_{i+1}:= X_{i} \cap f'' X_i \cap f^{-1}X_{i}$,
and $S:=\bigcap_{i\in \omega}X_i$.

The $X_n$ are decreasing, and 
$|\lambda\setminus X_n|<\kappa$ and thus
$|\lambda\setminus  (f'' X_n)|<\kappa$ for $n<\omega$.
Accordingly, $|\lambda\setminus S|<\kappa$.
We claim that $g:=f\restriction S$ is a permutation of $S$.
Clearly it is injective.
If $\alpha \in S$ then $\alpha \in X_n$ for all $n\in\omega$, 
so $f(\alpha)\in X_{n+1}$ for all $n$. So $g:S\to S$.
If $\alpha\in S$, then $\alpha\in X_{n+1}$ for all $n$, so $f^{-1}(\alpha)$
exists and is in $X_n$.
\end{proof}

Remark: For $\kappa=\lambda=\omega$, there are trivial automorphisms
that are not induced by ``proper'' bijections $f:\omega\to\omega$, e.g.
the automorphism $\phi$ induced by the almost permutation $n\mapsto n+1$.\footnote{A bijection $f:\omega\to\omega$ has infinitely many 
$n$ such that $f(n)\ne n+1$, and therefore an infinite
set $A$ such that $f''A$ is disjoint to $\{n+1:\, n\in A\}$.}

\medskip

We will investigate somewhere and densely trivial 
automorphisms. To simplify notation, we assume 
$\kappa=\lambda>\aleph_0$: 
\begin{definition}
($\lambda>\aleph_0$ regular.) 
Let $\phi:P^\lambda_\lambda\to P^\lambda_\lambda$ be an automorphism.
\begin{itemize}
\item 
$\phi$ is trivial on $A\in [\lambda]^\lambda$, if
there is an $f:A\to\lambda$ with 
$\phi([B])=[f''B]$ for all $B\subseteq A$.
\item 
$\phi$ is somewhere trivial, 
if it is trivial on some 
$A\in[\lambda]^{\lambda}$.
\item
$\phi$ is densely trivial, if for 
all $A\in[\lambda]^{\lambda}$ there is a $B\subseteq A$
of size $\lambda$ such that $\phi$ is trivial on $B$.
\end{itemize}
\end{definition}

Just as before it is easy to see that we can assume 
$f$ 
to be a full permutation:
\begin{fact} ($\lambda>\aleph_0$ regular.)
An automorphism 
$\phi:P^\lambda_\lambda\to P^\lambda_\lambda$ is trivial on $A\in[\lambda]^\lambda$
iff 
there is a bijection $f:\lambda\to\lambda$ 
such that
$\phi([B])=[f''(B)]$ for all
$B\subseteq A$.
\end{fact}




\begin{lemma}\label{lem:somewhereimpliesdensely}($\lambda>\aleph_0$ regular.)
If every automorphism of $P^\lambda_\lambda$ is somewhere trivial, then every automorphism of $P^\lambda_\lambda$ is densely trivial.
\end{lemma}
\begin{proof}
  Assume $\pi$ is an automorphism of $P^\lambda_\lambda$, and fix $A\in[\lambda]^\lambda$.
  If $A=^*\lambda$ and if $\pi$ is trivial on some $B$, then  $\pi$ is
  trivial on $B\cap A\subseteq A$, so we are done.
  So assume $A\ne^*\lambda$. 

  Pick some representative $\pi^*:\mathcal P(\lambda) \to \mathcal P(\lambda)$
  of $\pi$ such that 
  $\pi^*(A)$ and $\pi^*(\lambda\setminus A)$ 
  partition $\lambda$, and such that 
  $\pi^*(C)\subseteq \pi^*(A)$ for every $C\subseteq A$.
  Let $i: \lambda\setminus A \to A$ 
  and $j:\pi^*(\lambda \setminus A)\to \pi^*(A)$ both be bijective.
  Let $\pi'$ map $[D]$ to $[\pi^*(D\cap A)\cup j^{-1}\pi^*(i''(D\setminus A))]$.
  This is an automorphism of $P^\lambda_\lambda$, so it is trivial on some $D_0$. If $|D_0\cap A|=\lambda$, we are done,
  as $\pi'$ restricted to $D_0\cap A$ is the same as $\pi$
  and trivial.
  So assume otherwise. Then  
  $\pi'$ is trivial on the large set $D_0\setminus A$.
  Then $\pi$ is trivial on $i''(D_0\setminus A)\subseteq A$.
\end{proof}

\section{Under MA,  every automorphism is
trivial  for \texorpdfstring{$\omega_1\le\lambda<2^{\aleph_0}$}{ℵ₁≤λ<continuum}}\label{sec:ma}

\begin{theorem}\label{Matriv}
Assume $\aleph_0<\kappa\leq \lambda<2^{\aleph_0}$, 
$\kappa$ regular,
and  $\textnormal{\textrm{MA}}_{(=\lambda)}(\sigma\textnormal{\textrm{-centered}})$ holds. Then every automorphism of $P^\lambda_\kappa$ is trivial.
\end{theorem}

For the proof we will use that we can separate certain sets by closed sets.

A tree $T$ is a subset of  $2^{<\omega}$ such that 
$s\in T\cap 2^n$ and $m\le n$ implies $s\restriction m\in T$; for such a $T$ we set $\lim(T)=\{\eta\in 2^\omega:\, (\forall n\in\omega)\, \eta\restriction n\in T\}$. A subset of $2^\omega$ is closed iff it is of the form  $\lim(T)$
for some tree $T$.
\begin{lemma}\label{lem:thelem} 
Assume $\aleph_0<\theta\leq \lambda<2^{\aleph_0}$, 
$\cf(\theta)>\aleph_0$,
and  $\textnormal{\textrm{MA}}_{(=\lambda)}(\sigma\textnormal{\textrm{-centered}})$ holds.
    Assume $A_0, A_1$ 
    are disjoint subsets of $2^{\omega}$
    of size $\le\lambda$; 
    $|A_0|\ge\theta$.
    Then there is a tree $T_0$
    in $2^{<\omega}$ such that
    $|A_0\cap \lim(T_0)|\geq \theta$ and
    $A_1\cap \lim(T_0)=\emptyset$.
        
    If additionally $|A_1|\ge \theta$, we get an additional tree $T_1$ such that
    $|A_1\cap \lim(T_1)|\geq \theta$, $A_0\cap \lim(T_1)=\emptyset$, and
    $T_0\cap T_1\subseteq 2^n$ for some $n$.    
\end{lemma}
\begin{proof}[Proof of the lemma]
In the following we identify an $x\in 2^\omega$
with the according (infinite) branch $b$ in the
tree $2^{<\omega}$. So a branch $b$ can be in $A_0$
or in $A_1$ (or in neither; but not both, as $A_0$ and $A_1$
are disjoint).

We define a poset $Q$ as follows:
 A condition $q\in Q$ is a triple $(n_q,S_q,f_q)$, where
	\begin{itemize}
	    \item $n_q\in\omega$,
		\item $S_q$ is a tree in $2^{<\omega}$ 
		of the following form:
		$S_q$ is the union of 
		$2^{\le n_q}$ and finitely many (infinite) branches $\{b_j:\, j\in m\}$ for some $m\in\omega$, 
		each $b_j\in 
		A_0\cup A_1$, and $b_j\restriction n_q=b_k\restriction n_q$ implies
		($b_j\in A_i$ iff $b_k\in A_i$).
		
		So every $s\in S_q$ with $|s|> n_q$
		is either ``in $A_0$-branches''
		(i.e., there is one or more $b_j\in A_0$
		with $s\in b_j$), or ``in $A_1$-branches''
		(but not in both).

            Note that an $s\in S_q$  of length $n_q$
		is either in $A_0$-branches, or in $A_1$-branches, or in neither (but not in both).

		\item $f_{q}:  S_q\to 2$ 
		such that, for $i=0,1$, $f_{q}(s)=i$ whenever
		$s\in S_q$, $|s|\ge n_q$ 
		and $s$ is in $A_i$-branches.
    \end{itemize}

  The order on $Q$ is the natural one:  $q\le p$ if $n_q\ge n_p$, $S_q\supseteq S_p$
  and $f_q$ extends $f_p$.

  $Q$ is $\sigma$-centered
  witnessed by $(n_q,S_q,f_q)\mapsto (n_q, f_q\restriction 2^{\le n_q})$:
  If $p,q$ are in $Q$ with $n_p=n_q=:n$ and  $f_p\restriction 2^{\le n}=f_q\restriction 2^{\le n}$, then
  $(n, S_p\cup S_q, f_p\cup f_q)$ is a valid condition stronger than
  both $p$ and $q$.

  For $x\in A_i$, the set $D_x$ of conditions containing $x$ as branch is dense: Given $p\in Q$, 
  let $n_q\ge n_p$ be such that all $A_{1-i}$-branches 
  in $p$ split off $x$ below $n_q$; set 
  $S_q:=S_p\cup 2^{\le n_q}\cup x$;
  and set $F_q(s)=i$ for $s\in S_q\setminus S_p$.
  
  Similarly, for all $n\in\omega$,
  the set $D^*_n$ of conditions $q$
  with $n_q\ge n$ is dense as well. 
  
  By $\textrm{MA}_{(=\lambda)}(\sigma\text{-centered})$
  and $|A_i|\le \lambda$,  
  we can find a filter $G$ 
  which has nonempty intersection
  with each $D_x$ for $x\in A_0\cup A_1$
  as well as for each  $D^*_n$.
  So $F:=\bigcup_{p\in G} f_p$ is a total
  function from $2^{<\omega}$ to $2$;
  and for all $x\in A_i$
  there is an $n_x\in\omega$ such that 
  $m\ge n_x$ implies $F(x\restriction m)=i$.
  
  As $|A_0|\geq \theta$ and $\cf(\theta)>\aleph_0$
  we can assume that there is an $n^*_0$
  such that $n_x=n^*_0$ for $\theta$ many $x\in A_0$.
  If additionally $|A_1|\geq \theta$, we analogously 
  get an $n^*_1$ and set 
  $n^*:=\max(n^*_0,n^*_1)$; otherwise we set
  $n^*:=n^*_0$.
  We set
  $T_i^*:=\{s\in 2^{<\omega}: |s|\ge n^*,\,
  (\forall n^*\le k\le |s|)\, F(s\restriction k)=i\}$
  and generate a tree from it;
  i.e., we set $T_i:=T^*_i\cup \{s\restriction m:\, m<n^*, s\in T^*_i\}$.
  As we have seen above, $\lim(T_i)\cap A_i\ge \theta$ 
  for $i=0$ (and, if $|A_1|\ge \theta$, for $i=1$
  as well).
  Clearly $T_0\cap T_1\subseteq 2^{n^*}$;
  and $\lim(T_i)\cap A_{i-1}$ is empty, as for 
  any $x\in A_{i-1}$, cofinally many $n$
  satisfy $F(x\restriction n)=i-1$.
\end{proof}
 
\begin{proof}[Proof of the theorem]
Fix an automorphism $\pi$ of $P^\lambda_\kappa$
represented by some $\pi^*:\mathcal P(\lambda)\to\mathcal P(\lambda)$,
and let $\pi^{-1*}$ represent $\pi^{-1}$. 
We have 
to show that $\pi$ is trivial.
  
Fix an injective 
function $\eta:\lambda\rightarrow 2^\omega$.
Set 
\[C_n:=\{x\in 2^\omega:\, x(n)=0\}\text{ and }
\Lambda_{n}:=\eta^{-1}C_n=\{\alpha<\lambda:\eta(\alpha)(n)=0\}.\] 

Define $\nu:\lambda\to 2^\omega$ by
\[\nu(\beta)(n)=0\text{  iff }\beta\in \pi^*(\Lambda_n).\]
In the following, ``large'' means ``of cardinality
${\ge}\kappa$'', and ``small'' means not large.
We will show:

\newcommand{\jvkafirst}{($*_1$)}
\newcommand{\jvkasecond}{($*_3$)}
\newcommand{\jvkathird}{($*_2$)}
\newcommand{\jvkasecondsecond}{($*_4$)}
\newcommand{\jvkfive}{($*_5$)}
\newcommand{\jvksix}{($*_6$)}

\begin{itemize}
\item[\jvkafirst]
    $\pi^*(\eta^{-1}C) =^* \nu^{-1}C$ for $C\subseteq 2^\omega$ closed.
\item[\jvkathird]
    $Y\subseteq \lambda\text{ and }|Y|\ge \kappa$
     implies $|\nu''Y|\ge \kappa$.
\item[\jvkasecond]
If $A_0, A_1$ are disjoint subsets of $2^\omega$, $A_0\subseteq \nu''\lambda$ large,
then 
$\pi^{-1*}(\nu^{-1}A_0)\setminus \eta^{-1} A_1$ is large.
\item[\jvkasecondsecond]
If $A_0, A_1$ are disjoint subsets of $2^\omega$, $A_0\subseteq \eta''\lambda$ large,
then
$\pi^*(\eta^{-1}A_0)\setminus \nu^{-1} A_1$ is large.
\end{itemize}
(Note that~{\jvkathird} is the
only place where we use that $\kappa$ is regular.)
\\
Proof:
\begin{itemize}
    \item[{\jvkafirst}]
    $\pi^*(\eta^{-1}C_n) = \nu^{-1}C_n$ holds by definition of $\nu$.
As 
$\pi$ honors ${<}\kappa$-unions and complements, 
and as the $C_n$ generate the open sets, this equation (with $=^*$) holds whenever
$C$ is generated by ${<}\kappa$-unions and complements
from the open sets, in particular, if $C$ is closed.

\item[\jvkathird] Fix $x\in 2^\omega$. 
Then 
$\eta^{-1}\{x\}$ has at most one element  (as $\eta$ is injective), 
and $\eta^{-1} \{x\}=^* \pi^{-1*}\nu^{-1}\{x\}$ by~{\jvkafirst}.
I.e., 
$\nu^{-1}\{x\}$ is small. 
And $Y\subseteq \bigcup_{x\in \nu''Y}\nu^{-1}\{x\}$, so
as $\kappa$ is regular we get $|\nu''Y|\ge \kappa$.)

\item[\jvkasecond]
Using the previous lemma (with $\kappa$ as $\theta$)
we get a tree 
$T_0$ separating $A_0$ and $A_1$.
I.e., $\lim(T_0)\cap A_1=\emptyset$
and $X:=\lim(T_0)\cap A_0$ is large.
As $X\subseteq A_0\subseteq \nu''\lambda$, we get  that
$\nu^{-1} X$ is large. And $\nu^{-1} X=\nu^{-1}\lim(T_0)\cap \nu^{-1}A_0=^*\pi^*(\eta^{-1}\lim(T_0))\cap 
\nu^{-1}A_0$, the last equation by~{\jvkafirst}. This implies
$\eta^{-1}\lim(T_0)\cap \pi^{-1*}(\nu^{-1}A_0)$ is large,
and so $\pi^{-1*}(\nu^{-1}A_0)\setminus \eta^{-1} A_1$ is large.

\item[\jvkasecondsecond]
We get an analogous result when interchanging $\nu$ and $\eta$ and using
$\pi^*$ instead of $\pi^{-1*}$.


\end{itemize}
We claim that the following sets $N_i$ are all small:
\newcommand{\None}{N_{\ref{item:gg1}}}
\newcommand{\Ntwo}{N_{\ref{item:gg2}}}
\newcommand{\Nthree}{N_{\ref{item:gg3}}}
\newcommand{\Nfour}{N_{\ref{item:gg4}}}
\newcommand{\Nfive}{N_{\ref{item:gg5}}}
\newcommand{\Nsix}{N_{\ref{item:gg6}}}
\newcommand{\Nseven}{N_{\ref{item:gg7}}}
\newcommand{\Nnew}{N_{\ref{item:new}}}
\newcommand{\Nnewtwo}{N_{\ref{item:new2}}}
\begin{enumerate}
\item\label{item:gg5}  $\Nfive:=\{\alpha\in \lambda:\, (\lnot\exists \beta\in\lambda) \,\eta(\alpha)=\nu(\beta)\}$.
\item\label{item:gg1} $\None:=\{\alpha\in \lambda:\,  (\exists^{(\geq 2)} \beta\in \lambda)\,  \eta(\alpha)=\nu(\beta)\}$.
\item\label{item:gg2} $\Ntwo:=\{\beta\in \lambda:\, (\lnot \exists  \alpha\in \lambda)\, \eta(\alpha)=\nu(\beta)\}$. 
\end{enumerate}
Proof: 
\begin{itemize}
    \item[(\ref{item:gg2})] 
Assume $\Ntwo$ is large.
Set $A_0:=\nu'' \Ntwo$, which is large by {\jvkathird}; and $A_1:=\eta''\lambda$.
So $A_0$ and $A_1$ are disjoint, and by~{\jvkasecond}
$\pi^{-1*}\nu^{-1}A_0\setminus \eta^{-1}A_1$ is large,
but $\eta^{-1}A_1=\lambda$.

\item[(\ref{item:gg5})] Assume $\Nfive$ is large. 
Set $A_0=\eta''\Nfive$ (large, as $\neta$ is injective)
and $A_1:=\nu''\lambda$. So $A_0$ and $A_1$ are disjoint,
and by~{\jvkasecondsecond} $\pi^*(\eta^{-1}A_0)\setminus \nu^{-1} A_1$ is large,
but $\nu^{-1} A_1=\lambda$.

\item[(\ref{item:gg1})]  Assume that $\None$ is large.
For every $\alpha \in \None$, let $\beta_\alpha^0\neq \beta_\alpha^1$ in $\lambda$ be such that $\eta(\alpha)=\nu({\beta_\alpha^0})=\nu({\beta_\alpha^1})$. 
For $i\in\{0,1\}$, set $Y_i:=\{ \beta_\alpha^i:\ \alpha\in \None\}$
and $X_i:=\pi^{-1*}(Y_i)$ (without loss of generality disjoint),
and $A_i:=\eta'' X_i$.
So the $A_i$ are large and disjoint,
and we can find a tree $T_0$ such that
$A_0\cap\lim(T_0)$ is large, and $A_1\cap\lim(T_0)$ is empty.

As $A_0\subseteq \eta''\lambda$, this implies that the inverse
$\eta$-image of $A_0\cap\lim(T_0)$ is also large. I.e.,
$\eta^{-1}(A_0\cap\lim(T_0))=\eta^{-1}A_0\cap \eta^{-1}\lim(T_0)=^*  X_0\cap \pi^{-1*}\nu^{-1}\lim(T_0)$
is large (for the last equation we use \jvkafirst).
Therefore also $Y_0 \cap \nu^{-1}\lim(T_0)$ is large,
and so, by~{\jvkathird}, 
$\nu''(Y_0 \cap \nu^{-1}\lim(T_0))=\lim(T_0)\cap\nu'' Y_0 $ is large as well.

On the other hand $\lim(T_0)\cap A_1$ is empty, so
$0=^*\pi^*\eta^{-1} (\lim(T_0)\cap A_1)=^* \pi^*\eta^{-1}\lim(T_0)\cap 
 \pi^*\eta^{-1}A_1$.
 Using {\jvkafirst} for $\lim(T_0)$, and noting that $\pi^*\eta^{-1}A_1=Y_1$, this set is (almost) equal to 
 $Y_1\cap \nu^{-1}\lim(T_0)$ which therefore is also small, and so
 $ \lim(T_0)\cap \nu''Y_1$ is small.

 So we know that 
 $\lim(T_0)\cap \nu'' Y_0$ is large
 and $ \lim(T_0)\cap \nu''Y_1$ is small, but $\nu'' Y_0=\nu'' Y_1$,
 a contradiction.
\end{itemize}
Note that this implies:
\begin{itemize}
    \item[\jvkfive] $X\cap Y$ small implies $\nu''X\cap \nu''Y$ small, for $X,Y\subseteq\lambda$. 
\item[\jvksix] $\nu^{-1}\nu'' X=^*X$ for $X\subseteq \lambda$.
\end{itemize}
Proof: 
\begin{itemize}
    \item[\jvkfive]
Assume otherwise. Without loss of generality we can assume that
$X$ and $Y$ are disjoint, and by~(\ref{item:gg2})
that $\nu''X$ and $\nu'' Y$ both are subsets of $\eta''\lambda$.
Then $\nu''X\cap \nu''Y\subseteq  \eta''\None$
is small.
    \item[\jvksix]
Set $Y:=\nu^{-1}\nu''X\setminus X$.
Then $\nu''Y\subseteq N_2\cup N_3$ is small, and by {\jvkathird} $Y$ is small.
\end{itemize}

\medskip

Set $D:=\lambda\setminus (\Nfive\cup\None)$ and define 
$e:D\to\lambda$ such that  
$e(\alpha)$ is the (unique) $\beta\in\lambda$
with $\eta(\alpha)=\nu(\beta)$.
Clearly $e$ is injective.
We claim that $e$ generates $\pi$, i.e., that the following are small (where we can assume $X\subseteq D$):
\begin{enumerate}[resume]
\item\label{item:gg7} $\Nseven:=\pi^*(X)\setminus e''X$.
\item\label{item:gg6}
$\Nsix:=e''X\setminus \pi^*(X)$.
\end{enumerate}
Proof:
\begin{itemize}
    \item[(\ref{item:gg7})] Assume that 
$\Nseven$ is large. 
Set $Y=\pi^{-1*}(\Nseven)$, without loss of generality $Y\subseteq X$
and $\pi^*(Y)=\Nseven$.
So $\pi^*(Y)$ is disjoint from $e''Y$ (as it is even disjoint from $e''X$).
We set $A_0:=\nu''\pi^*(Y)$ and $A_1:=\nu''e''Y$, 
by~{\jvkfive} we can assume they are disjoint, 
and by~{\jvkathird} both are large ($e$ is injective).

By~{\jvkasecond},
$\pi^{-1*}(\nu^{-1}A_0)\setminus \eta^{-1} A_1$ is large.

$\eta^{-1}(A_1)=Y$,
as $\nu(e(\alpha))=\eta(\alpha)$ for all $\alpha\in D$.
And $\pi^{-1*}(\nu^{-1}A_0)=^*Y$ by definition and
{\jvksix}, a contradiction.


\item[(\ref{item:gg6})] The same proof works: This time we set $Y=e^{-1}\Nsix$; see that 
$\pi^*(Y)$ and $e''Y$ are disjoint
and large; set
$A_0:=\nu'' \pi^*(Y)$
and
$A_1:=\nu'' e'' Y$;
use~{\jvkasecond} to see that
$Y\setminus \eta^{-1}\nu'' e''Y=Y\setminus Y$
is large, a contradiction.\qedhere
\end{itemize}
\end{proof}

\section{For measureables, GCH implies a nontrivial automorphism}\label{sec:meas}

\begin{theorem}\label{thm:meas}
If $\lambda$ is measurable and
$2^\lambda=\lambda^+$, then there is a nontrivial automorphism of 
$P^\lambda_\lambda$.
\end{theorem}

\begin{proof}  
Let $\mathcal D$ be a normal ultrafilter on $\lambda$ 
and denote by $\mathcal I:=[\lambda]^\lambda\setminus \mathcal D$ its dual ideal restricted to sets of size $\lambda$.

 Since $2^\lambda=\lambda^+$, we can list all permutations of $\lambda$ as $\{e_\alpha: \alpha<\lambda^+\}$; 
 and analogously all elements of $\mathcal I $ as $\{X_\alpha: \alpha<\lambda^+\}$.
 
 We will construct, 
 by induction on $\alpha<\lambda^+$ a set 
 $A_\alpha\in \mathcal I$ and   a permutation $f_\alpha$ of $A_\alpha$, such that for  $\alpha<\beta$:
 \begin{enumerate}
   \item $A_\alpha\subseteq^*A_\beta$,
   \item $X_\alpha\subseteq A_{\alpha+1}$,
   \item\label{item:almostext} $f_{\alpha}(x)=f_{\beta}(x)$ for almost all $x\in A_{\alpha}\cap A_{\beta}$, 
   \item\label{item:differ} there is some $X\subseteq A_{\alpha+1}$  of size $\lambda$
   such that $e_\alpha''X$ and $f_{\alpha+1}''X$ are disjoint. 
\end{enumerate}
(Note that by $x\subseteq^*y $ we mean  $|y\setminus x|=\lambda$,
not $y\setminus x\in \mathcal I$; and the same for `almost all''.)
\\ 
The construction:
\begin{itemize}
    \item 

 Successor stages $\alpha+1$:
 Fix any  $B\in \mathcal I$ disjoint to $A_\alpha$ such that
 $A_\alpha\cup B\supseteq X_\alpha$. Set $C:=e_\alpha'' B\cap A_\alpha$.

 First assume that $|C|=\lambda$.
 Then set $A_{\alpha+1}=A_\alpha\cup B$ and let $f_{\alpha+1}$
 extend $f_\alpha$ by the identity on $B$.
 Then~(\ref{item:differ}) is witnessed by $X:=e_\alpha^{-1}C$. 

 So we assume $|C|<\lambda$. 
 Partition $B$ into large sets $B_0,B_1,B_2$ such that $e_\alpha''B_i$
 is disjoint to $A_\alpha$ for $i=0,1$.
 Set $A_{\alpha+1}:=A_\alpha\cup B\cup e_\alpha'' B$, and 
 define $f_{\alpha+1}$ on $B$ such that the restriction to $B_i$
 is a bijection op $e_\alpha'' B_{1-i}$ for $i=0,1$, and the  restriction to $B_2$ a bijection to 
 $e''B_2\setminus A$.
 Then~(\ref{item:differ}) is witnessed by $X:=B_0$. 
 

\item
  Limit stages $\delta$ of cofinality ${<}\lambda$: Let $\xi:= \cf(\delta)$ and choose $\langle\alpha_i: i<\xi\rangle$ a cofinal increasing sequence converging to $\delta$. The union $\bigcup_{i< \xi }A_{\alpha_i}$ is, by ${<}\lambda$ completeness, in $\mathcal I$.
  Remove $<\lambda$ many points
  to get a subset $A_\delta$ such that
  \begin{itemize}
      \item For all $i<j<\xi$,
      $f_i$ and $f_j$ agree on
      $A_{\alpha_i}\cap A_\delta$,
      \item For all $i<\xi$,
      $f_i\restriction (A_{\alpha_i}\cap A_\delta)$
      is a full permutation (we can do this as in
      Lemma~\ref{lem:almostisall}).
  \end{itemize}
   Then $f_\delta$, defined as the
   union of the $f_{\alpha_i}$, 
   is a permutation  of $A_\delta$ and almost extends each $f_{\alpha_i}$.

\item
  Limit stages $\delta$ of cofinality  $\lambda$:  We choose an increasing cofinal sequence $\langle \alpha_i: i<\lambda\rangle$ converging to $\delta$. By induction on  $i\in \lambda$  
  we construct $ A'_i=^*A_{\alpha_i}$, such that 
\begin{itemize}
	\item $A'_i\cap i=\emptyset$,
	\item The $f_{\alpha_i}$'s fully extend each other on the $A'_i$'s, i.e.,
	if $x\in A'_i\cap A'_j$ then $f_{\alpha_i}(x)=f_{\alpha_j}(x)$,
	\item $f_{\alpha_i}: A'_i \rightarrow A'_i$ is a ``full'' permutation.
\end{itemize}
  We can do this by removing from $A_{\alpha_i}$: the points less than 
  $i$, the points where $f_{\alpha_i}$ disagrees with some previous 
  $f_{\alpha_j}$ for any $j<i$;
  and by removing ${<}\lambda$ many points to get a full permutation.
  
  Now we can set $A_\delta$ and $f_\delta$
  to be the unions of $A'_{i}$ and $f_{\alpha_i}$, respectively, for $i<\delta$.
  Note that $A_\delta$ is in $\mathcal I$ (as it is a subset of 
  the diagonal union); 
  and $f_\delta$ is a permutation of
  $A_\delta$ satisfying~(\ref{item:almostext}).
\end{itemize}




		
Note that for all $X\subseteq \lambda$,
either $X\in\mathcal I$ or $\lambda\setminus X\in\mathcal I$
(but not both),
i.e., either $X$ or $\lambda\setminus X$
is $\subseteq^* A_\alpha$ for coboundedly many $\alpha<\lambda$.

This allows us to define the automorphism $\pi$ as follows:
For $X\in [\lambda]^\lambda$, 
\[\pi([X]):=\begin{cases} [f_\alpha''X] &\text{ if } X\in \mathcal I, X\subseteq^* A_\alpha \text{ for some }\alpha<\lambda^+\text{ (Case 1)}\\
[\lambda \setminus f_\alpha''(\lambda\setminus X) ]&\text{ if } X\notin \mathcal I,\lambda \setminus 
X\subseteq^* A_\alpha \text{ for some }\alpha<\lambda^+\text{ (Case 2)}.\end{cases}\]
Note that in Case 2, $\pi([X])=[(\lambda\setminus A_\alpha)\cup (A_\alpha\setminus f''_\alpha(A_\alpha\setminus X))]=[(\lambda\setminus A_\alpha)\cup f_\alpha''(X\cap A_\alpha)]$, as $f_\alpha''A_\alpha=^*A_\alpha$.

$\pi$ is well defined on $[\lambda]^\lambda$, as 
exactly one of $X$ or $\lambda\setminus X$ will eventually be 
$\subseteq^* A_\alpha$.

$\pi$ is an automorphism: $\pi([\emptyset])=\emptyset$.
$\pi$ honors complements: If $X$ is Case 1, then 
$\pi([\lambda\setminus X])$ is by definition (Case 2)
$[\lambda\setminus f_\alpha''(X)]$.
$\pi$ honors intersections $X\cap Y$:
This is clear if  both sets are the same Case.
Assume that $X$ is Case~1 and 
$Y$ Case~2. Then $X\cap Y\subseteq X$ is Case~1,
and for any $\alpha$ suitable for both $X$ and $Y$ we have
\[
  \pi([X])\wedge \pi([Y])=
  [f_\alpha'' X\cap ((\lambda\setminus A_\alpha)\cup f_\alpha''(Y\cap A_\alpha))] =
  [f_\alpha'' X \cap f''_\alpha(Y\cap A_\alpha)]=
  [f_\alpha'' (X\cap Y))].
\]

$\pi$ is not trivial: Every automorphism $e$
is an ${e_\alpha}$ for some $\alpha\in \lambda^+$;
and 
according to~(\ref{item:differ})
there is some $X_\alpha\subseteq A_{\alpha+1}$ 
(and therefore in $\mathcal I$) of size $\lambda$
such that $e_\alpha''X_\alpha$ is disjoint to
$f_{\alpha+1}'' X_\alpha$, a representative of 
$\pi([X_\alpha])$.
\end{proof}

\section{For inaccessible \texorpdfstring{$\lambda$}{λ}, all automorphisms can be densely trivial}\label{sec:forc}

In this section, we always assume the following
(in the ground model):
\begin{assumption}\label{asm:main}
$\lambda$ is inaccessible and $2^\lambda=\lambda^+$.
We set $\mu:=\lambda^{++}$.

\end{assumption}




In the rest of the paper, we will show the following:

\begin{theorem}\label{thm:main4}
 ($\lambda$ is inaccessible and $2^\lambda=\lambda^+$.)
 There is a $\lambda$-proper,
 ${<}\lambda$-closed,
 $\lambda^{++}$-cc poset $P$ 
 (in particular, preserving all cofinalities)
 that forced: $2^{\lambda}=\lambda^{++}$,
 and every automorphism of $P^\lambda_\lambda$ is densely trivial.
\end{theorem}
By Lemma~\ref{lem:somewhereimpliesdensely},
it is enough to show that every automorphism is somewhere trivial.


\newcommand{\xx}{\mathbf c}
\newcommand{\edn}{\epsilon_\mathrm{dn}}
\newcommand{\eup}{\epsilon_\mathrm{up}}

\subsection{The single forcing \texorpdfstring{$\bm Q$}{Q}}\label{sec:Q}

\begin{definition}
We fix a strictly increasing sequence $(\theta^*_\zeta)_{\zeta<\lambda}$
with $\theta^*_\zeta<\lambda$ regular
 and $\theta^*_\zeta>2^{|\zeta|}$.
\begin{itemize} 
\item 

Let $(I^*_{\zeta})_{\zeta\in\lambda}$
be an increasing interval partition of $\lambda$ such that $I^*_\zeta$ has size $2^{\theta^*_\zeta}$; and fix a bijection of
$I^*_\zeta$ and $2^{\theta^*_\zeta}$. 
Using this (unnamed) bijection, 
 we set
$[s]:=\{\ell \in I^*_\zeta: \ell>s\}$ for $s\in 2^{<\theta^*_\zeta}$.

So the $[s]$ are cones, i.e., the set of all branches in $I^*_\zeta$
extending $s$.


For $\zeta<\lambda$, we set $I^*({<}\zeta):=\bigcup_{\ell<\zeta} I^*_\ell$, and analogously
$I^*({\le}\zeta):=I^*({<}\zeta+1)$,
$I^*({\ge}\zeta):=\lambda\setminus I^*({<}\zeta)$,
and $I^*({\ge}\zeta,{<}\xi):=I^*({\ge}\zeta)\cap I^*({<}\xi)$.
\item A condition $q$ of the forcing notion $Q$ 
is a function with domain $\lambda$ such that,
for all $\zeta\in\lambda$,
$q(\zeta)$ is a partial function  from $I^*_\zeta$ to $2$, and such that
%
for a club-set $C^q\subseteq \lambda$ 
\begin{itemize}
\item if $\zeta\notin C^q$, then $q(\zeta)$ is total,
\item otherwise, the domain of $q(\zeta)$ is $I^*_\zeta\setminus [s^q_\zeta]$
for some $s^q_\zeta\in 2^{<\theta^*_\zeta}$.
\end{itemize}
$C^q$ and $s^q_\zeta$ are uniquely determined by $q$; 
and $q$ is uniquely determined by the partial function 
$\eta^q:\lambda\to 2$ defined as $\bigcup_{\zeta\in\lambda}q(\zeta)$.
\item $q$ is stronger than 
$p$ if $\eta^{q}$ extends $\eta^p$.

(This implies that $C^{q}\subseteq C^{p}$, and that $s^{q}_\zeta$ extends
$s^p_\zeta$ for all $\zeta\in C^{q}$.)
\end{itemize}
\end{definition}

The following is straightforward:
\begin{lemma}\label{lem:wrrq} 
$Q$ has size $2^{\lambda}$, is ${<}\lambda$-closed and adds a 
generic real 
$\neta:=\bigcup_{q\in G}\eta^q$
in $2^\lambda$.
\end{lemma}


\begin{proof}
${<}\lambda$-closure is obvious, but for later reference 
we would like to point out the ``problematic cases'':

Let $(p_i)_{i<\delta}$ be decreasing for a limit ordinal $\delta<\lambda$.

As a first approximation, set
$\eta^*:=\bigcup_{i<\delta}\eta^{p_i}$ (a partial function) 
and $C^*:=\bigcap_{i<\delta}C^{p_i}$
(a club set)
and $s^*_\zeta:=\bigcup_{i<\delta}s^{p_i}_\zeta\in 2^{\le \theta^*_\zeta}$ 
for $s\in C^*$. For $\zeta\notin C^*$, $\eta^*$ is indeed total on 
$I^*_\zeta$, and for $\zeta\in C^*$ the domain in $I^*_\zeta$
is $I^*_\zeta\setminus [s^*_\zeta]$.

The \textbf{problematic case} is when 
$s^*_\zeta$ is unbounded in $\theta^*_\zeta$.
(This can only happen if $\cf(\delta)= \theta^*_\zeta$, in particular for at most one $\zeta$.)
In this case we can just pick 
any extension $\eta^q$ of $\eta^{*}$
by filling all values in $I^*_{\le\zeta}$.
This gives the desired $q$, with $C^{q_\delta}=C^*\setminus {\zeta+1}$.
\end{proof}

\begin{remarks*}\mbox{}
\begin{itemize}
\item
The limits of ${<}\lambda$-sequences of conditions
are not 
``canonical'' if there are problematic $\zeta$'s, as 
we have to fill in arbitrary values.
\item
$\neta$ determines the generic filter, by 
$G=\{p\in Q:\,  \eta^p \subseteq \neta\}$.
This follows from the following facts: 
\begin{itemize}
    \item 
$p$ and $q$ are compatible (as conditions in $Q$)
iff $\eta^p$ and $\eta^q$
are compatible as partial functions and 
$X_{p,q}:=\{\zeta\in C^p:\, s^p_\zeta\text{ and }s^q_\zeta\text{ are incomparable}\}$
is non-stationary.
\item If $p,q$ are such that $X_{p,q}$ is stationary, 
then the set of conditions $r$ such that $\eta^r$ and $\eta^q$ are 
incompatible (as partial functions) is dense below $p$.
\end{itemize}
\end{itemize}
\end{remarks*}




\subsection{Properness of \texorpdfstring{$\bm Q$}{Q}: Fusion and pure decision.}\label{sec:Q2}
\begin{definition}
We say $q\le_{\xi} p$,
if $q\le p$, $\xi\in C^q$ and 
$q\restriction\xi=p\restriction \xi$.
\\
    $q\le^+_{\xi} p$ means $q\le_\xi p$ and
    $q(\xi)=p(\xi)$.
\end{definition}
(Note  the difference between $q\le^+_{\xi} p$ and
$q\le_{\xi+1} p$: The former does not require $\xi+1\in C^q$.)

\begin{lemma}\label{lem:Qfusion}
Let $\delta\le \lambda$ be a limit ordinal,
$\xi\in\lambda$ and $(q_i)_{i<\delta}$ a sequence in $Q$.
\begin{enumerate} 
\item\label{item:plainplus}
If $\delta<\lambda$ and 
$q_{j}<^+_{\xi} q_i$ for all $i<j<\delta$, then there is a 
 $q_\infty$ such that $q_\infty<^+_{\xi}q_i$ for all $i$.
\item\label{item:Qfusion}
If 
$q_{j}<_{\xi_i} q_i$ for $i<j<\delta$, where
$(\xi_i)_{i\in\delta}$ is
a strictly
increasing\footnote{For $\delta=\lambda$, it is enough that the $\xi_i$ converge to $\lambda$.
For $\delta<\lambda$, we use that the $\xi_i$ are increasing and that 
$\sup(\xi_i)\ge \cf(\delta)$.} sequence in $\lambda$,
then there is a (canonical) limit $q_\infty$ such that
$q_\infty<_{\xi_i}q_i$ for all $i$.
\end{enumerate}
\end{lemma}
\begin{proof}
    (\ref{item:plainplus}): We perform the same 
    construction as in the proof
    of Lemma~\ref{lem:wrrq}.
    If there is a problematic case $\zeta$,
    then $\zeta>\xi$ (as for $\zeta'\le \xi$ the 
    conditions $q_i(\zeta')$ are constant).
    We can then make $\eta^*$  total on $I^*(>\xi,\le\zeta)$.
    (It may not be enough to make it total on $I^*_\zeta$,
    as $C^*\setminus \{\zeta\}$ might not be club.)


    (\ref{item:Qfusion}): 
    Define $q_\infty(\zeta):= \bigcup_{i\in\delta} q_i(\zeta)$
    for $\zeta\in \lambda$.
    
    This is a non-total function (on $I^*_\zeta$) iff
    $\zeta\in C^{q_\infty}:=\bigcap_{i<\delta} C^{q_i}$,
    which is closed as intersection of closed sets, and also unbounded:
    If $\delta<\lambda$ because we have a small intersections of clubs,
    if $\delta=\lambda$ as it contains each $\xi_i$.
        
    There are no problematic cases: 
    If $\zeta$ is below some $\xi_i$, then $q_j(\zeta)$ is eventually constant.
    If $\zeta$ is above all $\xi_i$,
    which can only happen if $\delta<\lambda$,
    then $\cf(\delta)\le\delta\le\sup(\xi_i) \le \zeta< \theta^*_\zeta$.
\end{proof}

    
      
    

So $Q$ satisfies fusion; and we will now show that it also satisfies
``pure decision''; standard arguments then imply that $Q$
is $\lambda$-proper and $\lambda^\lambda$-bounding.

\begin{definition}\label{def:Qrelx} Let $\xi\in\lambda$, $q\in Q$.
    \begin{itemize}
    \item $\QPOSS(\xi):=2^{I^*({<}\xi)}$.
    So in the extension $V[G]$,
    for each $\xi$ there will be exactly one $x\in \QPOSS(\xi)$ compatible with
    (or equivalently: an initial segment of) the generic real $\neta$.
    We write ``$x\subseteq \neta$'' or ``$G$ chooses $x$'' for this $x$.
    \item 
    $\poss(q,\xi)$ is the set of $x\in \QPOSS(\xi)$ compatible with $ \eta^q$ (as partial functions), or equivalently:
    $x\in \poss(q,\xi)$ iff $\lnot q\Vdash x \nsubseteq\neta$.
    So $q$ forces that exactly one $x\in\poss(q,\xi)$ is chosen by $G$.
    \item 
    Let $\ntau$ be a name for an ordinal.
    We say that $q$ $\xi$-decides $\ntau$, if there is 
    for all $x\in \poss(q,\xi)$ an ordinal $\tau^x$ such that $q$ forces
    $x\subseteq\neta \rightarrow \ntau=\tau^x$.
\end{itemize}
\end{definition}
Note that for $p\in Q$ and $\zeta\in C^p$,
$q\le^+_{\zeta} p$ is equivalent to $\poss(q,\zeta+1)=\poss(p,\zeta+1)$,
while $q\le_\zeta p$ is equivalent to $\zeta\in C^q$ and $\poss(q,\zeta)=\poss(p,\zeta)$.

\begin{lemma}\label{lem:orelse}
  Assume $p\in Q$, $\zeta\in C^p$, 
  $x\in \poss(p,\zeta+1)$,
  and $r\le p$ extends\footnote{By which we mean $x\subseteq \eta^r$.} $x$.
  Then there is a $q\le^+_\zeta p$ 
  forcing: $x\subseteq\neta\rightarrow r\in G$.
  This condition is denoted by
  $r\vee (p\restriction \zeta+1)$.
\end{lemma}
    
\begin{proof}
  We set 
  $q(\ell)$ to be $p(\ell)$ for $\ell\le \zeta$, and 
  $r(\ell)$ otherwise.
  If $q'\le q$ forces $x\subseteq\neta$ then 
  $q'$ extends $x$ and thus $q'\le r$.
\end{proof}

\begin{corollary}\label{cor:Qpure} (``Pure decision'') 
    Let $\ntau$ be a name for an ordinal, $p\in Q$, and $\zeta\in C^p$.
    Then there is a $q\le^+_\zeta p$
    which $(\zeta+1)$-decides $\ntau$.
\end{corollary}


\begin{proof}
  Let $(x_i)_{i\in \delta}$ enumerate $\poss(p,\zeta+1)$, for some $\delta<\lambda$. Set $p_0=p$,
  and define a $\le^+_\zeta$-decreasing sequence 
  $p_j$ by induction on $j\le\delta$:
  For limits use Lemma~\ref{lem:Qfusion}(\ref{item:plainplus}),
  and for successors 
  choose some $r\le p_i$ deciding $\ntau$ with a stem extending $x_i$
  and set $p_{i+1}$ to 
  $r\vee p_i\restriction (\zeta+1)$.
\end{proof}
 
 


From fusion and pure decision we get 
bounding and $\lambda$-proper, via ``continuous reading of names''. This is a standard argument, and we will
not give it here; we will anyway prove a more ``general'' variant (for an iteration of $Q$'s), 
in Lemmas~\ref{lem:blb2} and~\ref{lem:esm1}.
\begin{fact}\mbox{}
   \begin{itemize}  
       \item $Q$ has continuous reading of names:
       If $\nsigma$ is a $Q$-name for a $\lambda$-sequence of ordinals, and $p\in Q$, then
       there is a $q\le p$ and there are $\xi_i\in \lambda$ 
       such that $q$  $\xi_i$-decides $\nsigma(i)$ for all $i\in\lambda$.   
       \item $Q$ is $\lambda^\lambda$-bounding. I.e., for every 
       name $\nsigma\in\lambda^\lambda$ and $p\in Q$
       there is an $f\in\lambda^\lambda$ and $q\le p$ such that 
       $q$ forces $f(i)>\nsigma(i)$ for all $i\in\lambda$.
       \item $Q$ is $\lambda$-proper. This means: If $N$ is a ${<}\lambda$-closed elementary submodel of $H(\chi)$ of size $\lambda$ containing $Q$, 
       with $\chi$ sufficently large and regular, and if $p\in Q\cap N$,
       then there is a $q\le p$ $N$-generic (i.e., forcing that
       each name of an ordinal which is in $N$ is evaluated
       to an ordinal in $N$).
   \end{itemize}
\end{fact}

For completeness, we also mention the following well-known fact 
(the proof is straightforward):
\begin{fact}\label{fact:boundingclub}
  Assume $\kappa$ is regular, and that the forcing notion
  $R$ is $\kappa^\kappa$-bounding. 
  Then $R$ preserves the regularity of $\kappa$, and
  every club-subset of $\kappa$ in the extension contains a ground model club-set.
\end{fact}

\subsection{The iteration \texorpdfstring{$\bm P$}{P}}
Let us first recall some well-known facts:
\begin{facts}\label{fact:trivial}
A ${<}\lambda$-closed forcing preserves cofinalities
${\le}\lambda$ and also the inaccessibilty
of $\lambda$.
\\
The ${\le}\lambda$-support iteration of
${<}\lambda$-closed forcings is ${<}\lambda$-closed.
\end{facts}


We will iterate the forcings $Q$ from the previous section in a
${<}\lambda$-closed ${\le}\lambda$-support iteration of length $\mu:=\lambda^{++}$:

\begin{definition}
Let $(P_\alpha,Q_\alpha)_{\alpha< \mu}$ be the ${\le}\lambda$-support
iteration such that each $Q_\alpha$ is the forcing $Q$ (evaluated in the $P_\alpha$-extension). We will write $P$ to denote the limit.
\end{definition}

\begin{remark*}
  One way to see that $P$ is proper
  is to use the framework of~\cite{MR2777755}.
  However, we will need an explicit form
  of continuous reading for $P$ anyway, which in turn gives
  properness for free.
\end{remark*}

\begin{definition}
Assume that $w\in [\mu]^{<\lambda}$ and $\xi\in\lambda$.
\begin{itemize}
    \item 
  $\bar\neta=(\neta_\alpha)_{\alpha\in\mu}$ is the sequence of
  $Q_\alpha$-generic reals  added by $P$.
\item 
  $\PPOSS(w,\xi):=2^{w\times I^*({<}\xi)}$.
  Exactly one $x\in \PPOSS(w,\xi)$ 
  is extended by $\bar\neta$, we write ``$x$ is selected by $G$,''
  or ``$x\triangleleft G$.''
  
\item 
  $\poss(p,w,\xi):=\{x\in\PPOSS(w,\xi):\, \lnot p\Vdash \lnot x\triangleleft G\}$. 
\item
  Let $\ntau$ be a name of an ordinal.
  $\ntau$ is $(w,\xi)$-decided by 
  $q$, if
  there are $(\tau^x)_{x\in \poss(q,w,\xi)}$ such that
  $q$ forces $x\triangleleft G\rightarrow \ntau=\tau^x$.
\end{itemize}
\end{definition}

Clearly, if 
$\ntau$ is $(w,\xi)$-decided by $q$,  and if
   $q'\le q$, $w'\supseteq w$ and $\xi'\ge \xi$, then
   $\ntau$ is $(w',\xi')$-decided by $q'$.

\begin{remark*}
  If $q\in P$ 
  $(w,\zeta)$-decides some $P_\alpha$-name $\ntau$, then the \emph{same} $q$ will
  generally \emph{not}  $(w\cap\alpha,\xi)$-decide $\ntau$ for any 
  $\xi$.\footnote{For example: 
  For a $p$-condition $Q$, 
  let $\odd^p$ be the set of odd elements of $C^p$
  (or any other unbounded subset $X$ of $C^p$ such that $C^p\setminus X$ is still club),
  and set $\odd^p_?:=\bigcup_{\zeta\in \odd^p}I^*_\zeta\setminus \dom(\eta^p)$. 
  Note that for any $x:\odd^p_?\to 2$,
  $\eta^p\cup x$ defines a condition in $Q$ (stronger than $p$).
  So if we fix any $p(0)\in P_1$, and define the $P_1$-name 
  $\ntau\in\{0,1\}$ to be 0 iff 
  $\neta_0\restriction\odd^{p(0)}_?$ is eventually constant to $0$,
  then $\ntau$ 
  cannot be $(\{0\},\zeta)$-decided by $p(0)$ for any $\zeta$.
  And if $p(1)$ is any condition with $p(0)\Vdash \eta^{p(1)}(0)=\ntau$, then $\ntau$ is $(\{1\},1)$-decided by $q:=(p(0),p(1))$.}
\end{remark*}

In the following, whenever we say
that $q$ $(w,\zeta)$-decides something, we 
implicitly assume that $w\in[\mu]^{<\lambda}$ and $\zeta\in\lambda$.
\begin{definition}Let $\nsigma$ be a $P$-name for a $\lambda$-sequence of ordinals.
\begin{itemize}
    \item  
  $q$ continuously reads $\nsigma$,
  if there are
  $(w_i,\xi_i)_{i\in\lambda}$
  such that $q$ $(w_i,\xi_i)$-decides $\nsigma(i)$
  for each $i\in \lambda$.
  \item 
  $P$ has continuous reading, if for each such $\nsigma$
  and $p\in P$ there is some $q\le p$ continuously reading $\nsigma$.
\end{itemize}
\end{definition}

The following is a straightforward standard argument:
\begin{fact}
        If $P$ has continuous reading, then it is $\lambda^\lambda$-bounding.
\end{fact}


As a first step towards pure decision, let us generalize the $\le_\zeta$-notation we
defined for $Q$:
\begin{definition}Let $p\in P$, $w\in[\mu]^{<\lambda}$ and
$\xi\in\lambda$.
    \begin{itemize}
        \item $p$ fits $(w,\xi)$, if $w\subseteq \dom(p)$ and 
  $p\restriction\alpha \Vdash \xi\in C^{p(\alpha)}$ for all $\alpha\in w$.
  \item $q\leq_{w,\xi} p$
        means:
        $q\leq p$, and for all $\alpha \in w$,
        $q\restriction\alpha$ forces $q(\alpha)<_\xi p(\alpha)$.
    \item $q\le^+_{w,\xi} p$ is defined analogously using
    $<^+_\xi$ instead of $<_\xi$.
  \end{itemize}
\end{definition}
Obviously $q\le^+_{w,\xi} p$ implies $q\leq_{w,\xi} p$; and 
$q\leq_{w,\xi} p$ implies that both $p$ and $q$ fit $(w,\xi)$.

\begin{remark*}
In contrast to the single forcing (or a product of such forcings), 
$q\le_{w,\xi}p$ (or $q\le^+_{w,\xi}p$) does \emph{not} imply $\poss(q,w,\xi)=\poss(p,w,\xi)$.\footnote{An example: 
$\dom(p)=\dom(q)=w=\{0,1\}$,
$\min(C^{p(0)})=\min(C^{q(0)})=\xi$,
and both 
$p(0)$ and $q(0)$ have trunk $a\in\QPOSS(\xi)$.
$p(0)$ forces that $p(1)=q(1)$,
that $\min(C^{p(1)})=\xi$ and that 
the trunk of
$p(1)$ is either $b$ or $c$ (elements of $\QPOSS(\xi)$);
both are possible with $p(0)$.
Now $q(0)\le^+_\xi p(0)$ decides that 
the trunk of $p(1)$ is $b$.
Then $q\le^+_{w,\xi}p$, and $(a,c)$ is in $\poss(p,w,\xi)\setminus \poss(q,w,\xi)$.
In particular 
$(a,c)\in\poss(p,w,\xi)$
but $p$ does not force that
$a\subseteq \neta_0$ implies $c\in\poss(p(1),\xi)$.}
More explicitly, setting $w=\{0,1\}$, it is possible that 
  $x\in\poss(p,w,\xi)$ but $p$ does not force that 
  $x(0)\subseteq \neta_0$ implies $x(1)\in\poss(p(1),\xi)$.
  (But see Section~\ref{sec:canonical}.)
\end{remark*}
\subsection{Continuous reading and properness of  \texorpdfstring{$\bm P$}{P}}

\begin{lemma}\label{lem:pluslimit}
  If $q_i$ is a $\leq^+_{w,\zeta}$-decreasing sequence of length $\delta<\lambda$,
  then there is an $r\leq^+_{w,\zeta}q_i$ for all $i<\delta$.
\end{lemma}

\begin{proof}
Set $\dom(r):=\bigcup_{i\in\delta}\dom(q_i)$, without loss of generality 
closed under limits.
By induction on $\alpha\in \dom(r)$ 
we know that $r\restriction\alpha\le q_i\restriction\alpha$ 
for all $i$, and define $r(\alpha)$ as follows:
If $\alpha\in w$, we know that the $q_i(\alpha)$ are $\le^+_\zeta$-increasing.
Using
Lemma~\ref{lem:Qfusion}(\ref{item:plainplus}), 
we pick some $r(\alpha)$ such that $r(\alpha)\le^+_\zeta q_i(\alpha)$ for all $i$.
If $\alpha\notin w$, we just pick any $r(\alpha)\le q_i(\alpha)$ for all $i$.
\end{proof}

It is easy to see that $P$ satisfies a version of fusion:
\begin{lemma}\label{lem:fusion}
    Assume $(p_i)_{i<\delta}$ is a sequence of length $\delta\le\lambda$,
    such that $p_{j}\leq_{w_i,\xi_i} p_i$
    for $i\leq j<\delta$, $w_i\in[\mu]^{<\lambda}$ increasing, $\xi_i\in\lambda$ strictly increasing. Set $w_\infty:=\bigcup_{i<\delta} w_i$,
    $\dom_\infty:=\bigcup_{i<\delta}\dom(p_i)$ and 
    $\xi_\infty:=\sup_{i<\delta}\xi_i$.
    If 
    $\delta=\lambda$, we additionally assume $w_\infty=\dom_\infty$.

    Then there is a limit $q_\infty$ 
    with $\dom(q_\infty)=\dom_\infty$
    such that
    $q_\infty \le_{w_i,\xi_i} p_i$ for all $i<\delta$.
    
    If $\delta<\lambda$, then 
    $q_\infty$ fits $(w_\infty,\xi_\infty)$.
\end{lemma}
(If $w_\infty=\dom_\infty$, then the limit $q_\infty$ is ``canonical''.)

\begin{proof}
   We define $q_\infty(\alpha)$ by induction on $\dom_\infty$.
   We assume that we already have 
   $q':=q_\infty\restriction \alpha$ which satisfies
   $q'\le_{w_i\cap \alpha, \xi_i}p_i$ for all $i<\delta$.
   
   Case 1: $\alpha\notin w_\infty$ (this can only happen if $\delta<\lambda$):
   We know that $q'$ forces that $(p_i(\alpha))_{i<\delta}$
   is a decreasing sequence, and we just pick some $q_\infty(\alpha)$
   stronger then all of them.
   
   Case 2: $\alpha\in w_\infty$:
   Let $i^*$ be minimal such that $\alpha\in w_{i^*}$.
   We know that $q'$ forces for all $i^*\le i<j<\delta$ that 
   $p_j(\alpha)<_{\xi_i}p_i(\alpha)$,
   so according to Lemma~\ref{lem:Qfusion}(\ref{item:Qfusion})
   there is a  limit $q_\infty(\alpha)<_{\zeta_i}p_i(\alpha)$
   (so in particular $q'\Vdash \zeta_i\in C^{q_\infty(\alpha)}$ for all $i\ge i^*$).
   
   Now assume $\delta<\lambda$.
   If $\alpha\in w_\infty$, then it is in $w_i$ for coboundedly many $i<\delta$.
   In other words,
   $p_j\restriction\alpha\Vdash \zeta_i\in C^{p_j(\alpha)}$ for coboundedly many $i\in\delta$ and all $j>i$,
   which implies $q_\infty\restriction\alpha\Vdash \xi_\infty \in C^{q_\infty(\alpha)}$.

\end{proof}

\begin{plemma}\label{plem:iubqwf}
  Let $p$ fit $(w,\zeta)$,
  $x\in \poss(p,w,\zeta+1)$, 
  and let $r\le p$ extend  $x$, i.e., $r\Vdash x\triangleleft G$.
  Then there is a $q\le^+_{w,\zeta} p$
  forcing that $x\triangleleft G$
  implies $r\in G$.
\end{plemma}
\begin{proof}
Set $\dom(q):=\dom(r)$. 
We define $q(\alpha)$
by induction on $\alpha\in \dom(q)$ and show inductively:
\begin{itemize}
    \item $q\restriction \alpha\le^+_{w\cap\alpha,\zeta}p\restriction\alpha$.
    \item $q\restriction\alpha\Vdash (x\restriction\alpha\triangleleft G_\alpha
\rightarrow r\restriction\alpha\in G_\alpha)$.
\end{itemize}
For notational convenience, we assume $\dom(p)=\dom(r)$ (by setting 
$p(\alpha)=\mathbb{1}_Q$ for any $\alpha$ outside the original domain of $p$).

Assume we already have constructed 
$q_0=q\restriction \alpha$.
Work in the $P_\alpha$-extension $V[G_\alpha]$ with $q_0\in G$.
\\
Case 1: $r\restriction \alpha\notin G_\alpha$.
Set $q(\alpha):=p(\alpha)$.
\\
Case 2: $r\restriction \alpha\in G_\alpha$. 
Then $r(\alpha)\le p(\alpha)$. If $\alpha\notin w$,
we set $q(\alpha):=r(\alpha)$;
otherwise we set $q(\alpha)$ to be 
$r(\alpha)\vee (p(\alpha)\restriction \zeta+1)$
as in Lemma~\ref{lem:orelse}.

If $\alpha\in w$, then in both cases we get 
 $q\restriction\alpha\vDash
 q(\alpha)\le^+_\zeta p(\alpha)$.
Also, if $G_{\alpha+1}$ selects $x\restriction (\alpha+1)$,
then at stage $\alpha$ we used, by induction, Case~2; so then
$r(\alpha)\in G(\alpha)$ as
$x(\alpha)\subseteq \neta_\alpha$. 
\end{proof}

We can iterate the construction for
all elements of $\poss(w, \zeta+1)$, which gives us:

\begin{lemma}\label{lem:pureplus}
  If $p$ fits $(w,\zeta)$ and $\ntau$ is a name for an ordinal, 
  then there is a  $q\le^+_{w,\zeta}p$ which $(w,\zeta+1)$-decides $\ntau$.
\end{lemma}

\begin{proof}
  We enumerate $\poss(p,w,\zeta+1)$ as $(x_i)_{i\in\delta}$.
  We start with $p_0:=p$. Inductively we construct $p_\ell$: 
  If at step $\ell$, if $x_\ell$ is not in $\poss(p_\ell,w,\zeta+1)$ any more, then
  we set $p_{\ell+1}:=p_\ell$. Otherwise,
  pick an 
  $r\le p_\ell$ that 
  decides $\ntau$ to be some $\tau^{x_\ell}$
  and extends $x_\ell$.
  Then apply~\ref{plem:iubqwf}
  to
  get $p_{\ell+1}\le^+_{w,\zeta} p_\ell$ which 
  forces that 
  $x_\ell\triangleleft G$
  implies $\ntau=\tau^{x_\ell}$. At limits use Lemma~\ref{lem:pluslimit}. 
\end{proof}

For the proof of Lemma~\ref{plem:blba} we will need a variant where the ``height'' $\zeta$ is not 
the same for all elements of $w$, more specifically:
\begin{plemma}\label{lem:pureplusplus}
  Assume that $p$ fits $(w,\zeta)$ and $p\restriction\alpha^*\Vdash \zeta^*\in C^{p(\alpha^*)}$, and that $\ntau$ is a name for an ordinal.
  Then there is a $q\le^+_{w,\zeta}p$ such that 
  $q\restriction\alpha^*\Vdash q(\alpha^*)\le^+_{\zeta^*}p(\alpha^*)$
  and there is a (ground model) set $A$ of size ${<}\lambda$ 
  such that $q\Vdash\ntau\in A$.
\end{plemma}

\begin{proof} This is just a notational variation of the previous proof.
  For notational simplicity we assume $\alpha^*\notin w$.
  
  First we have to modify~\ref{plem:iubqwf}:
  A candidate is a pair $(x,a)$ where 
  $x\in\PPOSS(w,\zeta)$ and $a^*\in\QPOSS(\zeta^*)$.
  Assume that $(x,a)$ is a candidate, that 
  $p\in P$ fits $(w,\zeta)$ and that 
  $p\restriction\alpha^*\Vdash \zeta^*\in C^{p(\alpha^*)}$,
  and assume that $r\leq p$ extends $(x,a)$, i.e.,
  $r\Vdash (x\triangleleft G\,\&\, a^*\subseteq \neta_{\alpha^*})$.
  Then there is a $q$ such that
  \begin{equation}\tag{$*$}\label{eq:wt33}
      q\le^+_{w,\zeta} p,\quad q\restriction\alpha^*\Vdash 
  q(\alpha^*)\le^+_{\zeta^*} p(\alpha^*),\, \  \text{and}\quad q\Vdash \bigl(
  (x\triangleleft G\, \&\, a^*\subseteq \neta_{\alpha^*})\, \rightarrow\,  r\in G\bigr).
  \end{equation}
  The same proof works, with the obvious modifications:
  
  When defining $q(\alpha)$, we inductively show:
\begin{itemize}
    \item $q\restriction \alpha\le^+_{w\cap\alpha,\zeta}p\restriction\alpha$
    and if $\alpha>\alpha^*$
    then $q\restriction\alpha^*\Vdash q(\alpha^*)\le^+_{\zeta^*} p(\alpha^*)$,
    \item 
$q\restriction\alpha\Vdash \bigl((x\restriction\alpha\triangleleft G_\alpha
\, \& \, a^*\subseteq \neta_{\alpha^*})
\rightarrow r\restriction\alpha\in G_\alpha\bigr)$, unless $\alpha\le \alpha^*$
in which case we omit the clause about $\alpha^*$.
\end{itemize}
Again, in the $P_\alpha$-extension we have:
\\
Case 1: $r\restriction \alpha\notin G_\alpha$.
Set $q(\alpha):=p(\alpha)$.
\\
Case 2: $r\restriction \alpha\in G_\alpha$. 
Then $r(\alpha)\le p(\alpha)$. If $\alpha\notin w\cup\{\alpha^*\}$,
we set $q(\alpha):=r(\alpha)$;
otherwise we set $q(\alpha)$ to be 
$r(\alpha)\vee (p(\alpha)\restriction \zeta+1)$
as in Lemma~\ref{lem:orelse}. 

Then we can show~\eqref{eq:wt33} as before.

\medskip

We then enumerate all candidates (there are ${<}\lambda$ many)
as $(x_\ell,a_\ell)$,
and at step $\ell$, if 
$(x_\ell,a_\ell)$ is compatible with $p_\ell$, use~\eqref{eq:wt33} to decide $\ntau$ to be some $\ntau^{\ell}$.
\end{proof}


We will now show that $P$ is $\lambda^\lambda$-bounding and proper.
We first give two preliminary lemmas that assume this is already the case
for all $P_{\beta'}$ with $\beta'<\beta$.

\begin{plemma}\label{plem:blba}
    Let $\beta\le\mu$, and 
    assume that $P_{\beta'}$ is $\lambda^\lambda$-bounding
    for all $\beta'<\beta$.
    
    Assume $p\in P_\beta$ fits $(w,\zeta)$, 
    $\tilde C\subseteq \lambda$ is club, and $\alpha^*<\beta$.
    
    Then there is a $q\le^+_{w,\zeta} p$
    and a 
    $\xi\in\tilde C$ 
    such that
    $q$ fits $(w\cup\{\alpha^*\},\xi)$.
    
    If additionally $\alpha^*\in\dom(p)$ and
    $p\restriction\alpha^*\Vdash \zeta^*\in C^{p(\alpha^*)}$ for some $\zeta^*\in\lambda$,
    then we can additionally get 
    $q\restriction\alpha^*\Vdash q(\alpha^*)\le^+_{\zeta^*} p(\alpha^*)$.
\end{plemma}

\begin{proof}
   For notational simplicity assume $\alpha^*\notin w$
   and
   $\min(\tilde C)>\max(\zeta,\zeta^*)$.
   By induction on $\alpha\le\beta$ we show
   that the result holds for all $w,\alpha^*$ 
   with $w\cup\{\alpha^*\}\subseteq\alpha$.

   \textbf{Successor case} $\alpha+1$: 
   Set $w_0:=w\cap \alpha$.
   
   By our assumption $P_\alpha$ is $\lambda^\lambda$-bounding, so every club-set in the $P_\alpha$-extension
   contains a ground-model club (see Fact~\ref{fact:boundingclub}). In particular,
   $C^{p(\alpha)}$ contains some ground-model $C^*$.
   By Lemma~\ref{lem:pureplus}  (or~\ref{lem:pureplusplus}, if $\alpha^*<\alpha$) 
   there is a
   $p'\le^+_{w_0,\zeta}p\restriction\alpha$
   (also dealing with $\alpha^*$, if  $\alpha^*<\alpha$)
   leaving only ${<}\lambda$ many possibilities for $C^*$. So
   we can intersect them all, resulting in $C'$. 
   Set $C'':=C'\cap \tilde C$.
   Apply the induction hypothesis in $P_\alpha$ to get 
   $q'\le^+_{w_0,\zeta} p'$ and $\xi$ in $C''$ such that
   $q'$ fits $(w_0,\xi)$ (and also $(\{\alpha^*\},\xi)$, if
   $\alpha^*<\alpha$).
   Set $q:=q'\cup\{(\alpha, p(\alpha))\}$,
   so trivially
   $q\le^+_{w,\zeta} p$
   (and, 
   if $\alpha=\alpha^*$, then 
   $q\restriction \alpha \Vdash q(\alpha)\le^+_{\zeta^*} p(\alpha)$),
   and $q$ fits $(w\cup\{\alpha\},\xi)$.

   
   \textbf{Limit case:} 
   If $w$ is bounded in $\alpha$ there is nothing to do. So assume $w$ is cofinal. 

   Set $\alpha_0:=\min(w\setminus\alpha^*)$
   and $w_0:=(w\cap\alpha_0)\cup\{\alpha^*\}$.
   Use the induction hypothesis in $P_{\alpha_0}$ 
   using 
   $(p\restriction \alpha_0,w_0,\zeta,\alpha^*,\zeta^*)$ as $(p,w,\zeta,\alpha^*,\zeta^*)$.
   This gives us some $p_0'\le^+_{w\cap\alpha_0,\zeta} p\restriction \alpha_0$ fitting 
   $(w_0,\zeta_0)$ and dealing with $\alpha^*$,
   for some $\zeta_0\in\tilde C$. Set $p_0:=p'\wedge p$.
   
   Enumerate $w\setminus w_0$  increasingly as $(\alpha_i)_{i<\delta}$, and 
   set $w_j:=w_0\cup \{\alpha_i:\ i<j\}$ for $j\le\delta$.
   
   We will construct $p'_i$ in
   $P_{\alpha_i}$ 
   and $(\zeta_i)_{i<\delta}$ a strictly 
   increasing sequence in $\tilde C$, 
   and we set $p_j:=p'_j\wedge p$ and will get:
   $p_\ell$ fits $(w_\ell,\zeta_\ell)$,and 
   $p_\ell \le^+_{w_i,\zeta_i} p_i$ 
   for all $i<\ell\le j$.

   For successors $\ell=i+1$, we use the induction hypothesis 
   in $P_{\alpha_{i+1}}$, using $(p_i\restriction \alpha_{i+1},w_i,\zeta_i,\alpha_i,\zeta)$ as $(p,w,\zeta,\alpha^*,\zeta^*)$.
   This gives us $p'_{i+1}\le^+_{w_i,\zeta_i} p_i\restriction \alpha_{i+1}$ and some 
   $\zeta_{i+1}>\zeta_i$ in $\tilde C$ 
   such that $p_{i+1}$ fits $(w_{i+1},\zeta_{i+1})$
   and $p_{i+1}\restriction \alpha_i\Vdash p_{i+1}(\alpha_i)\le^+_{\zeta} p_i(\alpha_i)$.
   
   For $j$ limit, we set $\zeta_j:=\sup_{i<j}\zeta_i$ (which is in $\tilde C$),
   and let $p_j$ be a limit of the 
   $(p_i)_{i<j}$. I.e., $\dom(p_j)=\bigcup_{i<j} \dom(p_i)$, and 
   for $\beta\in\dom(p_j)$ let $p_j(\beta)$ be as follows:
   If $\beta\notin w$, fix some condition
   $p_j(\beta)$ stronger than all $p_i(\beta)$.
   Otherwise, there is a minimal $i_0<j$ such that 
   $\beta\in w_{i_0}$, and $p_\ell(\beta)<^+_{\zeta_i} p_i(\beta)$
   for all $i_0\le i<\ell<j$.
   In that case let $p_j(\beta)$ be the (canonical) limit of the $(p_i(\beta))_{i_0\le i<j}$, and note that $\zeta_j\in C^{p_j(\beta)}$.
\end{proof}

\begin{plemma}\label{plem:blba5}
    Let $\beta\le\mu$, and 
    assume that $P_{\beta'}$ is $\lambda^\lambda$-bounding
    for all $\beta'<\beta$.
    
    Assume that $p\in P_\beta$ fits $(w,\zeta)$, 
    and $\nsigma$ is a $P_\beta$-name for a $\lambda$-sequence of ordinals.
    Then there is a $q\le^+_{w,\zeta} p$
    continuously reading $\nsigma$.
\end{plemma}

\begin{proof}
    Set $p_0:=p$, $\zeta_0:=\zeta$, $w_0:=w$.
    We construct by induction on $i< \lambda$ 
    $p'_i$, $p_i$, $\zeta_i$, $\alpha_i$ and $w_i$ as follows:
    \begin{itemize}
        \item Given $p_j$, $w_j$, and $\zeta_j$, pick $\alpha_j\in \dom(p_j)\setminus w_j$ by bookkeeping (so that in the end the domains of all conditions will 
        be covered).
        \item\label{item:wq38h3} Successor $j=i+1$: 
        Set $w_{i+1}:=w_i\cup\{\alpha_i\}$.
        Find $p'_{i+1}\le^+_{w_i,\zeta_i} p_i$ and $\zeta_{i+1}>\zeta_i$ 
        such that $p'_{i+1}$ fits $(w_{i+1},\zeta_{i+1})$ (using the previous preliminary lemma).
        \item Limit $j$: Let $p'_j$ be the canonical limit of the
        $(p_i)_{i<j}$, $\zeta_j:=\sup_{i<j}(\zeta_i)$, and $w_j:=\bigcup_{i<j} w_i$.
        Note that $p'_j$ fits $(w_j,\zeta_j)$.
        \item\label{item:wqtqw} In any case, given $p'_j$ we pick some $p_j\le^+_{w_j,\zeta_j} p'_j$
        which $(w_j,\zeta_j+1)$-decides $\nsigma(\zeta_j)$.
    \end{itemize}
    Then the limit $q$ of the $p_i$ continuously reads $\nsigma$.
\end{proof}

\begin{lemma}\label{lem:blb2}
    $P$ has continuous reading (and in particular is 
    $\lambda^\lambda$-bounding).
\end{lemma}

\begin{proof}
   Assume by induction that $P_{\beta'}$ is $\lambda^\lambda$-bounding for all $\beta<\beta'$.
   Then the previous lemma gives us that $P_\beta$ has continuous reading of names,
   and thus is $\lambda^\lambda$-bounding.
\end{proof}

The same construction shows $\lambda$-properness:
\begin{definition}\label{def:esm}
   Let $\chi\gg\mu$ be sufficiently large and regular.
   An ``elementary model'' is an
   $M\preceq H(\chi)$ of size $\lambda$ which is ${<}\lambda$-closed
   and contains $\lambda$ and $\mu$ (and thus $P$).
\end{definition}

\begin{lemma}\label{lem:esm1}
If $M$ is an elementary model containing $p\in P$,
then there is a $q\le p$ which is strongly $M$-generic in the following sense:
For each $P$-name $\ntau$ in $M$ for an ordinal,
$q$ $(w,\zeta)$-decides $\ntau$ via a decision function in $M$
(so in particular $q\Vdash \ntau\in M$).
\end{lemma}
(The decision function being in $M$ is equivalent to $w\subseteq M$, as $M$ is ${<}\lambda$ closed.)
\begin{proof}
   Let $\nsigma$ be a sequence of all $P$-names for ordinals that are in $M$.
   Starting with $p_0\in M$, perform the successor step of the previous construction 
   within $M$; as $M$ is closed the limits at steps ${<}\lambda$
   are in $M$ as well. 
   Then the $\lambda$-limit is $M$-generic.
\end{proof}

\subsection{Canonical conditions}\label{sec:canonical}
We will use conditions that ``continuously read themselves.''

\begin{definition}\label{def:canonic}
$p\in P$ is $(w,\zeta)$-canonical 
if $p$ fits $(w,\zeta)$ and 
$p(\alpha)\restriction (\zeta+1)$  is
$(w\cap\alpha,\zeta+1)$-decided
by $p\restriction \alpha$
for all $\alpha\in w$.
\end{definition}

\begin{facts}\label{fact:wrq325325} Let $p$ be canonical for $(w,\zeta)$. 
\begin{enumerate}
    \item\label{item:uqiwhrqwr} If $q\le^+_{w,\zeta} p$,
    then $q$ is canonical for $(w,\zeta)$
    and $\poss(p,w,\zeta+1)=\poss(q,w,\zeta+1)$
\item Let $x\in\poss(p,w,\zeta+1)$. There is a naturally defined $p\wedge x\le p$
such that $p\Vdash (p\wedge x\in G \leftrightarrow x\triangleleft G)$.
 $\{p\wedge x:\, x\in\poss(p,w,\zeta+1)\}$ is a maximal antichain below $p$.
 \item\label{item:uhwet}  Let 
    $x\in\poss(p,w,\zeta+1)$. In an intermediate $P_\alpha$-extension $V[G_\alpha]$
with $x\restriction \alpha\triangleleft G_\alpha$
the rest of $x$, i.e., $x\restriction [\alpha,\mu]$,
is compatible with $p/G_\alpha$ in the quotient forcing.

Or equivalently: If $r_0\leq p\restriction\alpha$
in $P_\alpha$ extends $x\restriction \alpha$, then there
is an $r\le r_0$ extending $x$.
\end{enumerate}
\end{facts}

\begin{definition}\label{def:canonicreading}
Assume $p\in P$, 
and $\nsigma$ is a $P$-name for a $\lambda$-sequence of ordinals.
Let 
$E\subseteq \lambda$ be a club-set 
and 
$\bar w=(w_\zeta)_{\zeta\in E}$
an increasing 
sequence in
$[\mu]^{<\lambda}$.

$p$ canonically reads $\nsigma$
as witnessed by $\bar w$ if  the following holds:
\begin{itemize}
     \item 
     $\dom(p)=\bigcup_{\zeta\in E} w_\zeta$.
     \item 
     $p$ is $(w_\zeta,\zeta)$-canonical  for all $\zeta\in E$.
     \item 
     $p\restriction \alpha\Vdash C^{p(\alpha)}=E\setminus (\zeta'_\alpha)$
     for some (ground model) $\zeta'_\alpha$.
       \item 
       $\nsigma\restriction I^*({\le}\zeta+1)$ is 
       $(w_\zeta,\zeta+1)$-decided by $p$
       for all $\zeta\in E$.
   \end{itemize}
\end{definition}
If $\sigma$ is the constant $0$ sequence (or any sequence in $V$),
we just say ``$p$ is canonical'' (as witnessed by $\bar w$).

\begin{lemma}\label{lem:canonicdense}
  For $p$, $\nsigma$ as above,
  there is a $q\le p$ canonically reading  $\nsigma$.

If $p\in P_\alpha$ and $\nsigma$ is a $P_\alpha$-name for some $\alpha<\mu$,
then $q\in P_\alpha$.
\end{lemma}


\begin{proof}
  We just have to slightly modify the proof of
  Lemma~\ref{plem:blba5}.
  
    We will construct 
    $p_j$, $\xi_j$ and $\alpha_j$ 
    by induction on $j\in \lambda$,
    setting $w_j:=\{\alpha_i:i<j\}$,
    such that for $0<j<k$ the following holds:
    \begin{itemize}
        \item $p_{k}\le^+_{w_j,\xi_j} p_j$.
        \item $p_{j}$ is $(w_j,\xi_j)$-canonical.
        \item $p_{j}$ $(w_j,\xi_j+1)$-decides $\nsigma\restriction I^*({\le}\xi_j+1)$.
        \item In $p_{k}$,
        for $\alpha_j\in w_k$,
        $\{\zeta_i:\, j<i<k\}$
        is (forced to be) an initial segment of
        $C^{p_k(\alpha_j)}$.
        \item The $\alpha_j$ are chosen (by some book-keeping)
        so that $\{\alpha_i:\, i\in\lambda\}=\bigcup_{i\in\lambda}\dom(p_i)$.
    \end{itemize}
    Then the limit of the $p_j$ is as required,
    with $E=\{\xi_i:\, i\in \lambda\}$
    and, for $\zeta=\xi_j$ in $E$, we use $w_j$ as $w_\zeta$.

    Set $p_0\le p$ such that $|\dom(p_0)|=\lambda$,
    and set $\xi_0:=0$. 
    Assume we already have $p_i,\alpha_i$ for $i<j$
    (so we also have $w_j$). 
    \begin{itemize}
      \item For $j$ limit, let $s$ be a limit of $(p_i)_{i<j}$,
      and set $\xi_j:= \sup_{i<j}\xi_i$. Note that $s$ fits
      $(w_j,\xi)$.
      \item Successor case $j=i+1$:
        Find $s_0\le^+_{w_i,\xi_i} p_i$ and $\xi_{j}>\xi_i$ 
        such that $s$ fits $(w_{j},\xi_{j})$.
        (As in Lemma~\ref{plem:blba}. Recall that $w_j=w_i\cup\{\alpha_i\}$.)
        
        Strengthen   $s_0$ to $s\le^+_{w_i,\xi_i}$ so that:
        \begin{itemize}
            \item $s$ still fits $(w_{j},\xi_{j})$,
            \item the trunk at $\alpha_i$ has length $\xi_j$, 
        i.e.,
        $s\restriction \alpha_i\Vdash \min(C^{s(\alpha_i)})=\xi_j$),
        \item for $\alpha_{i'}$, $i'<i$,
        there are no elements in $C^{s(\alpha_{i'})}$
        between $\xi_i$ and $\xi_j$.
        \end{itemize}
      \item 
      Construct $s^{*}\restriction\alpha$ by recursion on $\alpha\in w_j$,
        such that $s^{*}\restriction \alpha\le^+_{w_j\cap\alpha,\xi_j} s\restriction\alpha$
        and $s^{*}\restriction \alpha$
        $(w_j\cap\alpha,\xi_j+1)$-decides $s(\alpha)\restriction (\xi_j+1)$ (which is the same as $s^*(\alpha)\restriction (\xi_j+1)$).
        This gives $s^{*}\le^+_{w_j,\xi_j} s$.
      \item         
        Find $p_j\le^+_{w_j,\xi_j} s^{*}$ which 
        $(w_j,\xi_j+1)$ decides $\nsigma\restriction I^*({\le}\xi+1)$.
      \item Choose $\alpha_j\in \dom(p_j)\setminus w_j$ by bookkeeping.
     \qedhere
    \end{itemize}
        
        
    
    
\end{proof}




\begin{facts}\label{fact:countingcanonics}
    \begin{enumerate}
    \item\label{item:trivial134} If a $P_\beta$-name $\name x\subseteq \lambda$
    is continuously read (by some $P_\beta$-condition $p$),
    and $\cf(\beta)>\lambda$, then there is an
    $\alpha<\beta$ such that:
    $p\in P_\alpha$, and
    $\name x$ is already a 
    $P_\alpha$-name (formally:
    there is a $P_\alpha$-name $\name y$ such that 
    $p\Vdash \name x=\name y$).
    
        \item\label{item:countingcanonics} There are at most $|\alpha|^\lambda\le \lambda^+$
        many pairs\footnote{Depending on the formal definition, we could/should add ``modulo equivalence'', i.e.,
        there is a ${\le}|\alpha|^\lambda$-sized set $Z$ of such pairs
        such that whenever $p$ canonically reads $\name y$ in
        $P_\alpha$ then there is a $\name x$ such that
        $(p,\name x)\in Z$ and $p\Vdash \name x=\name y$.} $(p,\name x)$ such that 
        $p$ canonically reads $\name x$ in $P_\alpha$.       
    \end{enumerate}
\end{facts}

\subsection{\texorpdfstring{$\mathbf\Delta$}{Δ} systems}\label{ss:stat}

In this section we define $\Delta$-systems and show that such systems exist, which we 
will in the indirect proofs of Lemmas~\ref{lem:notinbeta} and~\ref{lem:wrtqewqw}.

In Section~\ref{sec:fixit} we will then fix a specific $\Delta$-system for the rest of the paper.

From now on, we 
assume that $p_*$ forces
\begin{equation}\label{eq:basiasm}
\npi: \mathcal P(\lambda)\to\mathcal P(\lambda)
\text{ represents the automorphism }\nphi: P^\lambda_\lambda \to P^\lambda_\lambda,
\end{equation}
and we
set, for $\beta\in\mu$, 
\[
   \na_\beta:=\npi(\neta_\beta),
\]
where, as usual, we identify $\neta_\beta\in 2^\lambda$ with $\neta_\beta^{-1}\{1\}\subseteq \lambda$. 

Note that, other than $\neta_\beta$, $\na_\beta$ is a priori not a $P_{\beta+1}$-name (but see Section~\ref{sec:inbeta}).

We also fix a $P$-name for a representation
of the inverse automorphism $\nphi^{-1}$.
Abusing notation, we call it $\npi^{-1}$.

With  $S^\mu_{\lambda^+}$ we denote the stationary subset of 
$\mu$ consisting of ordinals with cofinality $\lambda^+$.


\begin{definition}\label{def:esmsystem}
   Let $S\subseteq S^{\mu}_{\lambda^+}$ be stationary,
   $\chi\gg\mu$ sufficiently large and regular, and
   $z\in H(\chi)$.
   ``An elementary $S$-system'' (using parameter $z$) is a sequence 
   $(M_\beta,p_\beta)_{\beta\in S}$ such that, for each $\beta\in S$,
    $M_\beta$ is an elementary model (as in Definition~\ref{def:esm})
    and contains $z$, $\beta$, $p_*$, $\nphi$, $\npi$ and $\npi^{-1}$,
    and $p_\beta\in P\cap M_\beta$ canonically  reads $\na_\beta$ witnessed by some $(w^{p_\beta}_\zeta)_{\zeta\in E^{p_\beta}}$,
    which $E^{p_\beta}\subseteq\lambda$ club (cf.\ Def.~\ref{def:canonicreading}).
\end{definition}


By a simple $\Delta$-system argument we can make 
an $S$-system homogeneous:
\begin{definition}\label{def:delta}
$(M_\beta, p_\beta)_{\beta\in S}$ 
forms a ``$\Delta$-system'',
if  
$\bar M,\bar p$ is an elementary $S$-system with parameter $z$,
 and
   is homogeneous in the 
    following sense: For $\beta$ and $\beta_1< \beta_2$ in $S$, we get:
\begin{enumerate}
  \item $M_{\beta_1}\cap M_{\beta_2}\cap \mu$
  is constant. We call this set the ``heart'' and, abusing notation, denote it with
  $\Delta$. Obviously $\Delta\supseteq \lambda$,
  $\Delta\supseteq \dom(p_*)$, $\lambda^+\in\Delta$, etc.
  \item $M_{\beta}\cap {\beta}=\Delta$.
  So in particular $\beta$ is the minimal element of $M_\beta$ above $\Delta$.
  All the non-heart elements
  of $M_{\beta_2}$ are above all elements of $M_{\beta_1}$. I.e.,
  $\sup(M_{\beta_1}\cap \mu)<\beta_2$.
  \item There is an $\in$-isomorphism $h^*_{\beta_1,\beta_2}: M_{\beta_1}\to 
  M_{\beta_2}$, mapping $\beta_1$ to $\beta_2$,
  $p_{\beta_1}$ to $p_{\beta_2}$,
  $\na_{\beta_1}$ to $\na_{\beta_2}$
  and fixing $\lambda,\mu,\nphi,\npi$ as well as each $\alpha$ in $\Delta$.
\end{enumerate}    
\end{definition}
Note that this implies that the continuous reading of $\na_\beta$ 
works the same way for all $\beta$.  
In particular the $E^{p_{\beta}}$ are that same $E$
for all $\beta$; and if $F^\beta_\zeta$ is the function
  mapping $\PPOSS(w^{p_\beta}_\zeta,\zeta+1)$ to the value of
  $\na_\beta\restriction I^*({\le}\zeta+1)$
  (for $\zeta\in E$), then
  $h^*_{\beta_1,\beta_2}(F^{\beta_1}_\zeta)=F^{\beta_2}_\zeta$ and in particular
  $h^*_{\beta_1,\beta_2}(w^{p_{\beta_1}}_\zeta)=w^{p_{\beta_2}}_\zeta$; i.e., they are the same apart 
  from shifting coordinates above $\Delta$.

\begin{lemma}\label{lem:deltaexists}\mbox{}
Assume $S\subseteq S^\mu_{\lambda^+}$ is stationary.
\begin{itemize} 
    \item For every $z\in H(\chi)$ 
    and $(p'_\beta)_{\beta\in S}$
    there are $M_\beta$ and $p_\beta\le p'_\beta$ such that $\bar M,\bar p$ is an $S$-system
    with parameter $z$.
    \item If $\bar M,\bar p$ is an $S$-system then there 
    is an $S'\subseteq S$ stationary 
    such that $(M_\beta,p_\beta)_{\beta\in S'}$ is a $\Delta$-system on $S'$.
\end{itemize}
\end{lemma}

\begin{proof}
The first item is trivial, using the fact that everything can be read canonically.

Using $2^\lambda=\lambda^+$,
a standard $\Delta$-system argument (or: Fodor's Lemma argument)
lets us thin out $S$ to some $S^2$
so that $(M_\beta\cap \mu)_{\beta\in S^2}$ satisfies (1--3).
For $\beta\in S^2$ let 
$\iota_\beta:M_\beta\cup \{M_\beta\}\to H(\lambda^+)$
be the transitive collapse, and 
assign to $\beta$ the tuple of the $\iota_\beta$-images of 
the following objects:
\begin{itemize}
    \item $M_\beta$, $p_\beta$, $\na_\beta$, $\mu$, $\nphi$, $\npi$,
and $E^{p_\beta}$.
    \item For $\zeta\in E^{p_\beta}$, the object $w^{p_\beta}_\zeta$,
    \item For 
    $\zeta\in E^{p_\beta}$ and $\gamma\in w^{p_\beta}_\zeta$,
    the object 
$F^{p_\beta}_\gamma$.
\end{itemize}
Again, there are $|H(\lambda^+)|^\lambda<\mu$ many possibilities, so the objects are
constant on a stationary $S'\subseteq S^2$.

For $\alpha<\beta$ in $S'$, we define $h^*_{\beta_1,\beta_2}:=
\iota^{-1}_{\beta_2}\circ \iota_{\beta_1}$.
(Note that $\iota_{\beta_1}(\alpha)=\iota_{\beta_2}(\alpha)$ for $\alpha\in\Delta$.)
\end{proof}

  


So in particular 
if we have a $\Delta$-system on $S$, then
$p_\beta\restriction\sup(\Delta)=p_\beta\restriction\beta\in  M_\beta$ is the same for all $\beta\in S$, 
and outside
of $\Delta$ the domains of the $p_\beta$ are disjoint
for $\beta\in S$. In particular we get:

\newcommand{\blubb}[1]{\begin{equation}\parbox{0.9\textwidth}{#1}\end{equation}}

\begin{fact}\label{fact:joint}
For a $\Delta$-system with domain $S$, and 
$A\subseteq S$ of size ${\le}\lambda$,
the union of the $(p_{\beta})_{\beta\in A}$
is a condition in $P$ (and stronger than each $p_\beta$).
\end{fact}

Whenever $r \in P_\beta\cap M_\beta$
(as is the case for 
$r=p_\beta\restriction \beta$),
we know that $r\in P_\alpha$ for $\alpha\in \Delta$
(as $M_\beta$ knows that $\beta$ has cofinality $\lambda^+$).

Instead of ``$r\in P_\alpha$ for some $\alpha\in \Delta$'' we will sometimes just state the weaker but shorter $r\in P_{\sup(\Delta)}$.

\begin{remark*}
This is an important effect also for some names.
Generally, a $P_\beta$-name in $M_\beta$ is of 
course not a $P_\alpha$-name for any $\alpha<\beta$
(just take the $P_\beta$-generic filter $G_\beta$).
However, as we will explicitly state in 
Lemma~\ref{lem:ogottogott},
 such names for subsets of $\lambda$
 are, modulo some condition, $P_\alpha$-names for some $\alpha\in \Delta$ and independent of $\beta$. 
 In 
 the specific case of the 
 $P_\beta$-name $p_\beta(\beta)$ we do not have to
 increase the condition:
\end{remark*}




\begin{definitionandlemma}
$\tilde p:=p_\beta(\beta)$
is a $P_{\sup(\Delta)}$-name independent of $\beta\in S$.
\end{definitionandlemma}

\begin{proof}
$p_\beta(\beta)\restriction\zeta+1$
is $(w^{p_\beta}_\zeta,\zeta+1)$-determined 
for cofinally many $\zeta \in E$, where
$w^{p_\beta}_\zeta\in [\beta]^{<\lambda}$
is a subset of $M_\beta$.
So 
$w^{p_\beta}_\zeta\subseteq \Delta$,
and the isomorphisms between the $M_\beta$ guarantee that each
$w^{p_\beta}_\zeta$ is the same, and that
$p_\beta(\beta)\restriction\zeta+1$ is 
decided the same way.
So $\tilde p$ is a $P_\gamma$-name for $\gamma=\sup(w^{p_\beta}_\zeta)_{\zeta\in E}$.
This $\gamma$ is independent of $\beta\in S$, and is in $\Delta$.
So $\tilde p$ is actually a $P_\alpha$-name for some $\alpha\in \Delta$; and certainly a $P_{\sup(\Delta)}$-name.     
\end{proof}




For later reference we note:
\begin{lemma}\label{lem:notinbeta}
For all but non-stationary many $\beta$,
$p_*$ forces $\na_\beta\notin V_\beta$.
\end{lemma}
(Here, $V_\beta$ denotes the $P_\beta$-extension of the ground model.)

\begin{proof}
Assume that $p_\beta\le p_*$ forces that 
$\na_\beta=\name x_\beta$ for a $P_\beta$-name $\name x_\beta$ for all $\beta\in S^*$ stationary.
We can also assume that $p_\beta$ canonically reads
$\na_\alpha$. Pick $M_\beta$ containing $p_\beta$
and $S\subseteq S^*$ such that $(M_\beta,p_\beta)_{\beta\in S}$
is a $\Delta$-system, where we can assume 
(or get from homogeneity) that 
$h^*_{\beta_0,\beta_1}(\name x_{\beta_0})=\name x_{\beta_1}$. So  the $\name x_{\beta}$
are $P_\beta$-names in $M_\beta$ and therefore $P_{\sup(\Delta)}$-names,
and are the same for all $\beta$.
Choose $\beta_1>\beta_0$ in $S$.
So $p_{\beta_0}\wedge p_{\beta_1}$ force that
$\na_{\beta_0}=\name x=\na_{\beta_1}$, which
contradicts the injectivity of $\nphi$ and
the fact that $\neta_{\beta_0}\ne \neta_{\beta_1}$.
\end{proof}

\subsection{Preservation of cofinalities, catching canonical names}

\begin{corollary}\label{cor:cof}
  $P$ is $\lambda^{++}$-cc and preserves all cofinalities.
\end{corollary}
\begin{proof}
  Cofinalities ${\le}\lambda$ are preserved as $P$ is ${<}\lambda$-closed.
  
  Cofinality $\lambda^+$ is preserved by properness:
  Assume that it is forced by $p$ that $\kappa$ has a cofinal $\lambda$-sequence $\bar{\name\alpha}:=(\name\alpha_i)_{i\in\lambda}$.
  Then there is an elementary model $M$ containing $p$
  and $\bar {\name\alpha}$.
  If $q\le p$ is $M$-generic, and $G$ a $P$-generic
  filter containing $q$, then $\name\alpha_i[G]\in M$
  for all $i<\lambda$, so $M\cap\kappa$
  is a cofinal subset of $\kappa$ of size $\lambda$ in the ground model.
  
  Cofinality $\ge\lambda^{++}$ is preserved as $P$ has the  $\lambda^{++}$-cc, which we have shown in a very roundabout way
  with the fact about $\Delta$-systems:
  If $(p'_\alpha)_{\alpha\in \mu}$ are arbitrary conditions, then
  $(M_\beta,p_\beta)$ form a $\Delta$-system from some $p_\beta<p'_\beta$ 
  and stationary
  $S$, and any two 
  (in fact, ${\le}\lambda$ many) 
  $p_\beta$ are compatible for $\beta\in S$.
\end{proof}

\begin{remark}
This shows that $P$ is
$(\mu,\lambda)$-Knaster, i.e.,
for every $A\in [P]^\mu$ there is a
$B\in [A]^\mu$ which is $\lambda$-linked.
\end{remark}

The $\lambda^{++}$-cc also implies:
For every name $\name x$ for a subset of 
$\lambda$ (or of $\lambda^+$) 
there is a $\beta<\mu$
and a $P_\beta$-name $\name y$ such that 
the empty condition forces that $\name x = \name y$.

Given $\alpha<\mu$, there are ${<}\mu$ many
pairs $(p,\name x)$ where $p$ canonically reads $\name x\subseteq \lambda$ in $P_\alpha$, see Fact~\ref{fact:countingcanonics}(\ref{item:countingcanonics}).
So there is a $g(\alpha)<\mu$
such that for each such $p,\name x$,
both $\npi(\name x)$ and $\npi^{-1}(\name x)$
are equivalent (modulo the empty
condition) to some $P_{g(\alpha)}$-name. 
Let $C^*\subseteq\mu$ be the club set with
$(\zeta\in C^*\, \&\, \alpha<\zeta)\ \rightarrow\ g(\alpha)<\zeta$.

Given a $\Delta$-system on $S$ we can restrict it
to a $\Delta$-system on $S\cap C^*$; so we will
assume from now on that each $\Delta$-system 
we consider satisfies $S\subseteq C^*$.


To summarize: 

\begin{lemma}\label{lem:ogottogott}    
    \begin{enumerate}
        \item\label{item:whbgrqq} If $\beta\in S$,
        $p\in P_\beta$ and
        $\name x$ a $P_\beta$-name for a subset of $\lambda$, then there is an $\alpha<\beta$
        and a $q\le p$ canonically reading 
        $\name x$, $\ntau(\name x)$, $\ntau^{-1}(\name x)$ as $P_\alpha$-names.
        
        More explicitly: There is a $P_\alpha$-name
        $\name y$ which is canonically read by $q$
        such that $q\Vdash \name y=\name x$.
        (And analogously for $\ntau(\name x)$
            and $\ntau^{-1}(\name x)$ instead of $\name x$.)
        \item\label{item:qwiuhwqt} If additionally $p\le p_\beta\restriction\beta$
        in $P_\beta$ and $(p,\name x)\in M_\beta$,
        then we can additionally get:
        $\name x$, $\npi(\name x)$ and $\npi^{-1}(\name x)$
        are 
        $P_\alpha$-names 
        in $M_\beta$ independent of $\beta\in S$.
        
        More explicitly: 
        Let $\name y$ be as above
        (for $\name x$). Then $\alpha\in\Delta$,  
        $q$ and $\name y$ are in $M_\beta$, and if $\beta'\in S$ and 
        $h:=h^*_{\beta,\beta'}$, then 
        $h$ acts as identity on $\alpha$, $q$, and $\name y$, and 
        ($M_{\beta'}$ knows that) $q\Vdash \name y=h(\name x)$. 
        (And analogously for $\ntau(\name x)$
            and $\ntau^{-1}(\name x)$ instead of $\name x$.)
    \end{enumerate}
\end{lemma}
\begin{proof}
(\ref{item:whbgrqq}):
    Use Lemma~\ref{lem:canonicdense} to get a $q_1\in P_\beta$
    canonically reading $\name x$.
    And if $\beta\in S$ then $\cf(\beta)=\lambda^+$,
    so $\dom(p)$ is bounded by some $\alpha'<\beta$ and, by Fact~\ref{fact:countingcanonics}(\ref{item:trivial134}), 
    $q_1\in P_{\alpha_1}$ for some $\alpha'\le \alpha_1<\beta$.
    As $\beta\in C^*$, $\npi(\name x)$ and $\npi^{-1}(\name x)$
    are $P_\beta$-names. So repeat the same
    argument to get $q\le q_1$ in $P_\alpha$ canonically
reading all three subsets of $\lambda$.   

(\ref{item:qwiuhwqt}): Apply (\ref{item:whbgrqq}) inside
$M_\beta$. As $\alpha\in\beta\cap M_\beta$, we get $\alpha\in\Delta$.
As $q$ canonically reads 
itself as well as $\name y$, we know that 
$h$ does not change $q$ and $\name y$.
As $h$ is an isomorphism, we know that 
$h(q)=q$ forces that $h(\name x)=h(\name y)=\name y$.
\end{proof}
\subsection{Majority decisions}

For any $(a_1,a_2,a_3)$ with $a_i\in \{0,1\}$ there is a
$b\in \{0,1\}$ such that 
$b=a_i$ for at least two $i\in \{1,2,3\}$.
We write $b=\majority_{i=1,2,3}(a_i)$.

Similarly, if 
$f_1$, $f_2$, $f_3$ are functions $A\to 2$
we write $\majority_{i=1,2,3}(f_i)$
for the function $A\to 2$ that maps $\ell$
to $\majority_{i=1,2,3}(f_i(\ell))$.


The following is a central point of
the whole construction:  

\begin{lemma}\label{lem:blal0}
Let $(M_\alpha,p_\alpha)_{\alpha\in S}$ be a $\Delta$-system.
Pick $\beta_0<\beta_1<\beta_2<\beta_3$ in $S$. 
\begin{enumerate}
    \item\label{item:bla101}  $p_*$ forces: If
    $\neta_{\beta_0}=^*\majority_{i=1,2,3}(\neta_{\beta_i})$,
    then  
    $\na_{\beta_0}=^*\majority_{i=1,2,3}(\na_{\beta_i})$.
    \item\label{item:bla102}     
    Let $s=\bigwedge_{i<4} p_{\beta_i}$.
    Recall that $s(\beta_i)$ is the same $P_{\sup(\Delta)}$-name
    called $\tilde p$ for all $i$.
    We can strengthen $s$ by strengthening, for 
    $i=1,2,3$, the condition
    $s(\beta_i)=\tilde p$ 
    to some $P_{\beta_0+1}$-names  $r_i\le \tilde p$
    (without changing $C^{\tilde p}$) 
    such that the resulting condition forces
    $\neta_{\beta_0}=\majority_{i=1,2,3}(\neta_{\beta_i})$.

    (We do not have to strengthen $s(\beta_0)$ for this,
    i.e., we can use $r_0:=\tilde p$.)
\end{enumerate}
\end{lemma}
We describe this by ``$(r_i)_{i<4}$ honors majority''.

Recall that $\nu_1=^*\nu_2$ denotes that $\nu_1(\ell)=\nu_2(\ell)$
    for all but ${<}\lambda$ many $\ell\in\lambda$.

\begin{proof}
   (\ref{item:bla101}) Identifying $2^\lambda$ with $P(\lambda)$, we 
   have 
   $\majority_{i=1,2,3}{f_i}= (f_1\cap f_2)\cup (f_2\cap f_3)\cup (f_1\cap f_3)$ for any tuple
   $(f_i)_{i=1,2,3}$.
As $\npi$ represents an automorphism,
we get
    $\npi(\majority_{i=1,2,3}(f_i))=^* 
\majority_{i=1,2,3}(\npi(f_i))$. Apply this to $f_i:=\neta_{\beta_i}$.


(\ref{item:bla102})
Work in the $P_{\beta_0+1}$-extension.
Recall  $\tilde p:=p_{\beta_0}(\beta_0)$.
So both $\tilde p$ and $\neta_{\beta_0}$
are already determined, and 
$\neta_{\beta_0}$ 
extends $\eta^{\tilde p}$. Set $r_0:=\tilde p$.

Set $s_1:=(0,0)$, $s_2:=(0,1)$, $s_3:=(1,0)$.
For $\zeta\in C^{\tilde p}$ and  $i=1,2,3$, 
we define $r_i(\zeta)\supseteq \tilde p(\zeta)$ as follows:
\begin{equation}\label{eq:rtqw}
\text{Extend }s^{\tilde p}_\zeta\text{ by }s_i,\text{i.e., } 
s^{r_i}_\zeta:=(s^{\tilde p}_\zeta)^\frown s_i; 
\text{ and set }r_i(\zeta)(\ell):=\neta_\beta(\ell)
\text{ for }\ell\in [s^{\tilde p}_\zeta]\setminus [s^{r_i}_\zeta].
\end{equation}
So $\eta^{r_i}$ 
agrees on its domain with $\neta_{\beta_0}$, 
and each $\ell\in \lambda$ is in 
$\dom(\eta^{r_i})$ for at least two $i\in\{1,2,3\}$.
Accordingly, an extension by a generic filter $G$ with $r_i\in G(\beta_i)$ for all $i<4$
will satisfy $\neta_{\beta_0}=\majority_{i=1,2,3}(\neta_{\beta_i})$. (We do not even have to assume
that any $p_\beta\in G$.)
\end{proof}



\begin{remark}
    Let $p'_{\beta_1}$ be the condition where we strengthen
    $p_{\beta_1}(\beta_1)$ to $r_1$.
    Note that $p'_{\beta_1}$
    is not in $M_{\beta_1}$, as 
    $\beta_0\notin M_{\beta_1}$
    and $r_1$ is defined using $\neta_{\beta_0}$.
    Similarly (basically the same):
    $r_1[G_{\beta_1}]\notin M_{\beta_1}[G_{\beta_1}]$,
    even if we assume that $G_{\beta_1}$ is $M_{\beta_1}$-generic. 
    But generally we will not be interested in $M_\beta$-generic conditions or extensions (we needed generic conditions only in Lemma~\ref{lem:esm1}, which
    in turn is needed for Corollary~\ref{cor:cof}).
    And while usually most conditions 
    we consider can be constructed within (and therefore
    will be elements of) some $M_\beta$, 
    this is generally not required
    (an example are the $s_i$'s in the following Lemma).
\end{remark}

The same proof works if we do not start with the
$p_\beta$ but with any stronger conditions, as long as they still
``cohere'' in the way that the $p_{\beta_i}$ cohere:
\begin{lemma}\label{lem:8hhgewuhwegw3t}
Let $(M_\alpha,p_\alpha)_{\alpha\in S}$ be a $\Delta$-system,
$\beta_0<\beta_1<\beta_2<\beta_3$ in $S$, and 
    $s_{i}\le p_{\beta_i}$ for $i=0,1,2,3$ such that:
    \begin{itemize}
        \item $\dom(s_i)\subseteq  M_{\beta_i}$
        \item $s^*:=s_i\restriction \beta_i$ is the same for all $i$,
        \item $s^*$ forces that 
    the $s_{i}(\beta_i)$ are the same for all $i$.

    (In the usual sense: The $s_{i}(\beta_i)$ are continuously read from 
    generics below $\beta_0$ in the same way for each $i<4$.)
    \end{itemize}
    Then there is condition stronger than all $s_{i}$
    forcing that 
    $\neta_{\beta_0}=\majority_{i=1,2,3}(\neta_{\beta_i})$ and thus $\na_{\beta_0}=^*\majority_{i=1,2,3}(\na_{\beta_i})$.  
\end{lemma}

\subsection{\texorpdfstring{
$\underaccent{\bm{\sim}}{\bm{a}}_{\bm{\beta}}$}{aᵦ} is  in the \texorpdfstring{\bm{$\beta+1$}}{β+1}-extension}\label{sec:inbeta}

We now show that $\na_\beta$ can be assumed to be a 
$P_\beta$-name.

The following definitions, in particular everything concerning
the notion of coherence, is used only in this section.
In the rest of the paper, we will use from this section only Lemma~\ref{lem:wrtqewqw}, i.e., the fact that 
$\na_\beta\in V_{\beta+1}$.



\begin{remark*}
Why do we introduce this (rather annoying) notion
of coherence? Well, we would like to simultaneously
construct something like $s_i\le p_{\beta_i}$ where
each $s_i$ ends up in $M_{\beta_i}$. 
We cannot directly do this
in $M_{\beta_0}$, as $M_{\beta_0}$ does not know about, e.g., $\beta_1$. So instead, we construct 
four different $s_i'\le p_{\beta_0}$ in $M_{\beta_0}$
in such a way (a ``coherent'' way) and use
$s_i:= h^*_{\beta_0,\beta_i}(s_i')$.
\end{remark*}

Let us for now (until Lemma~\ref{lem:wrtqewqw})
fix an arbitrary $\Delta$-system
$(M_\beta,p_\beta)_{\beta\in S}$  
as well as $\beta_0<\beta_1<\beta_2<\beta_3$ in $S$. 
For notational convenience, set
\[
\beta:=\beta_0.
\]

\begin{definition}

\begin{itemize}
    \item 
    $\bar q=(q_i)_{i<4}$ in $M_\beta$ is called coherent, if
    each $q_i$ is stronger than $p_{\beta}$ and 
    $q_i\restriction(\beta+1)$ is the same for all $i<4$.
    \item 
    If $\bar q$ is coherent, then
    $\bigwedge_{i<4}h^*_{\beta,\beta_i}(q_i)$
    is a valid condition
    in $P$, and we call it $q^*$.
    
    I.e., $q^*$ is the union of the copies of $q_i$ in $M_{\beta_i}$; and the
    copy for $q_0$ is just $q_0$.

    $r\in P$ is called coherent, if $r=q^*$
    for some coherent $\bar q\in M_{\beta}$.
\end{itemize}
\end{definition}

\begin{facts*}
\begin{itemize}

\item
The $p_{\beta_i}$ are coherent, more correctly:
\\
The condition $\bigwedge_{i\in 4}p_{\beta_i}$ is coherent; equivalently: The tuple
$\big(h^{*-1}_{\beta,\beta_i}(p_{\beta_i})\big)_{i<4}$ is coherent.
\item 
Any coherent $r$ is stronger than $\bigwedge_{i<4}p_{\beta_i}$.
\item
If $\bar q$ is coherent,
$r_i\le q_i$ in $M_{\beta}$ for $i<4$, 
and $r_i\restriction \beta_i$ is the same for all $i<4$,
then $\bigwedge_{i<4}h^*_{\beta,\beta_i}(r_i)$  is
(a valid condition and) 
compatible with $q^*$.
\item
$r\in P$ is coherent
iff: $\dom(r)\subseteq \bigcup_{i<4} M_{\beta_i}$,
$r\restriction (\mu\cap M_{\beta_i})\in M_{\beta_i}$ is stronger than $p_{\beta_i}$,
and each $r(\beta_i)$ is forced to be the same 
condition. 

In that case, 
$r=q^*$ for 
$q_i:=h^{*-1}_{\beta,\beta_i}(r_i)$ and
$r_i:=r\restriction (\mu\cap M_{\beta_i})$
\end{itemize}
\end{facts*}

\begin{lemma}\label{lem:itismajafterall}
  If $r$ is coherent,
  then it can be strengthened\footnote{to a condition that will generally not be coherent} to
  force\footnote{Here we write $\beta_0$ instead of $\beta$ to stress the interaction
with $\beta_1,\dots,\beta_3$, but recall that $\beta:=\beta_0$.} $\na_{\beta_0}=\majority_{i=1,2,3}\na_{\beta_i}$.
\end{lemma}

\begin{proof}
  This follows from Lemma~\ref{lem:8hhgewuhwegw3t},
  using $s_i:=r\restriction (\mu \cap M_{\beta_i})$.
\end{proof}

\begin{definition}\mbox{}
\begin{itemize}
  \item $\bar w=(w_i)_{i<4}$ is coherent, if 
  $w_i\in [\mu]^{<\lambda}$ is in $M_{\beta}$ 
  and $w_i\cap (\beta+1)$ is independent of $i$.
  
  In the following  we always assume that 
  $\bar q$ and $\bar w$ are coherent.
\item 
  $\bar q$ fits $(\bar w,\zeta)$, if
  each $q_i$ fits $(w_i,\zeta)$.
  \item $\bar q$ is $(\bar w,\zeta)$-canonical, if
  each $q_i$ is $(w_i,\zeta)$-canonical.
  \item $\bar r\le^{+}_{\bar w,\zeta}\bar q$
  means: $\bar r$ is coherent, and $r_i\le^+_{w_i,\zeta}q_i$
  for all $i<4$. 
  \item 
  $\bar x=(x_i)_{i<4}$ is defined to be in $\poss(\bar q,\bar w,\zeta)$
if $x_i\in\poss(q_i,w_i,\zeta)$ and $x_i\restriction \beta$ is independent of $i$. 
Such a $\bar x$ will be called coherent possibility.

(Note that the $x_i(\beta)$ in a coherent possibility can be different for different $i<4$. Also note that
such a $\bar x$ is automatically in $M_\beta$, which is ${<}\lambda$-closed.)
\end{itemize}
\end{definition}

Note that if $\bar r\le^{+}_{\bar w,\zeta}\bar q$ and 
$\bar q$ is $(\bar w,\zeta)$-canonical, then 
$\bar r$ and $\bar q$ have the same coherent $(\bar w,\zeta+1)$-possibilities,
see Fact~\ref{fact:wrq325325}(\ref{item:uqiwhrqwr}).

Several of the previous constructions 
result in coherent 4-tuples
when applied to coherent 4-tuples.
In particular:

\begin{lemma}\label{lem:many}\mbox{}
\begin{enumerate}
    \item\label{item:trivial} Assume $(\bar q^j)_{j\in\delta}$  is a sequence 
    of coherent 4-tuples
    such that, for each $i<4$, the $i$-part
    $(q^j_i)_{j\in\delta}$ satisfies the assumptions of
    Lemma~\ref{lem:pluslimit}.
    
    Then for each $i$, the lemma (in $M_\beta$) gives us a limit
    $r$, which we call $q^\delta_i$.

    We can choose the $q^\delta_i$ 
    so that they form a 
    coherent 4-tuple.
    \item\label{item:trivial2}
    The same applies to Lemma~\ref{lem:fusion}.
    I.e., we can get a coherent fusion limit from a $\lambda$-sequence
    of coherent tuples.
    \item\label{item:trivial3}
    Assume $\bar p$ fits $(\bar w,\zeta)$, and $\alpha_i\in \mu$
    such that $w'_i:=w_i\cup\{\alpha_i\}$ is coherent.
    Then there is a $\xi>\zeta$ and a $\bar q\le^+_{\bar w,\zeta}\bar p$
    which fits $(\bar w',\xi)$ and is $(\bar w',\xi)$-canonical.    
    \item\label{item:annoying} 
    Assume $\bar q$ is  coherent and (for simplicity)
    $(\bar w,\zeta)$-canonical with $\beta\in w_i$ (which is independent of $i<4$), and
    $\ntau_i$ are names of ordinals. 
    Then there is an  
    $\bar r\le^{+}_{\bar w,\zeta}\bar q$ 
    such that $\bar\ntau$ is $(\bar w,\zeta+1)$-decided by $\bar r$.

    By this we mean that 
$\ntau_i$ is $(w_i,\zeta+1)$-decided by $r_i$ for all $i<4$.
\end{enumerate}
\end{lemma}

\begin{proof}
%
For the first items, we just have to look at the proofs of the according lemmas
(For ~(\ref{item:trivial3}) this is~\ref{plem:blba} and~\ref{plem:blba5}) and
note that 
coherent input gives us coherent output.
In the following we will prove~(\ref{item:annoying}).
We work in $M_\beta$.

Enumerate all coherent possibilities as $(\bar x_k)_{k\in K}$.
Set $\bar r^0:=\bar q$. We now construct $\bar r^{k+1}$
from $\bar r:=\bar r^k$ where we assume 
$\bar r^k\le^{+}_{\bar w,\zeta}\bar q$.
\begin{itemize}
    \item Find $s_0$ stronger than $r_0$
    and extending $x_0$, deciding $\ntau_0$.
    
    \item 
    $s^*:=(s_0\restriction\beta)\wedge r_1$ 
    is stronger than $r_1$, as $\bar r$ is coherent.
    Strengthen $s^*(\beta)=r_1(\beta)=r_0(\beta)$
    to $s_0(\beta)$, but replace 
    the trunk with $x_1(\beta)$.
    Then $s^*\restriction\beta$ forces that
    $s^*(\beta)\le r_1(\beta)$,
    as 
    $x_1\restriction \beta= x_0\restriction\beta$ and
    as
    $x_1(\beta)$ is guaranteed to be possible, because $r_1$ is canonical.
    Further strengthen $s^*$ (above $\beta$)
    to extend (the rest of) $x_1$; 
    and then strengthen the whole condition once
    more to decide $\ntau_1$. Call the result $s_1$.
    \item 
    Do the same for $i=2$, starting with
    $s_1$, resulting in $s_2$, and then for
    $i=3$, starting with $s_2$, resulting in some $s_3$.
    %
    %
    
    So $s_i\le r_i$ extends $x_i$ and
    decides $\ntau_i$, and 
    $s_3\restriction\beta\le s_i\restriction\beta$
    and $s_3(\beta)$ is stronger than $s_i(\beta)$
    ``above $\zeta+1$''.
    \item 
    We define $r'_i\le r_i$ as follows:
    $\dom(r'_i)=(\dom(s_3)\cap \beta)\cup \dom(s_i)$.
    We define $r'_i(\alpha)$ inductively
    such that 
    $r'_i\restriction\alpha \le^+_{w_i\cap\alpha,\zeta} r_i$
    forces that 
    $x_i\restriction \alpha \triangleleft G$
    implies $s_i\restriction\alpha\in G$.
    \begin{itemize}
        \item 
        For $\alpha\le \beta$:
        
        
        If $s_3\restriction \alpha\notin G_\alpha$, set $r'_i(\alpha)=r_i(\alpha)$. Assume otherwise.
        So $s_3(\alpha)$ is defined
        and stronger than 
        $r_i(\alpha)=r_3(\alpha)$.
        If $\alpha\notin w_i$
        (which implies $\alpha<\beta$), set
        $r'_i(\alpha)=s_3(\alpha)$.
        Otherwise, use $s_3(\alpha)\vee (r_3(\alpha)\restriction \zeta+1)$, as in Lemma~\ref{lem:orelse}.
        \item For $\alpha>\beta$, we do the same 
        but we use $s_i$
        instead of $s_3$. In more detail:
        
        If $s_i\restriction \alpha\notin G_\alpha$, set $r'_i(\alpha)=r_i(\alpha)$. Assume otherwise.
        If $\alpha\notin w_i$, set
        $r'_i(\alpha)=s_i(\alpha)$.
        Otherwise, use $s_i(\alpha)\vee (r_i(\alpha)\restriction \zeta+1)$.
    \end{itemize}
    We can use this $\bar r'$ as $\bar r^{k+1}$: 
    It is coherent, $\bar r'\le^{+}_{\bar w,\zeta}\bar r^k$,
    and 
    $r'_i$  decides $\ntau_i$
    assuming $x_i\triangleleft G$.
\qedhere
\end{itemize}
\end{proof}

Coherent tuples
$\bar q$  naturally define a $P$-condition $q^*$.
However, we 
have to assume that $\bar q$ is canonical to guarantee
that coherent $\bar q$ possibilities correspond to $q^*$-possibilities:
\begin{lemma}
Assume $\bar q$ and $\bar w$ coherent. We set $w^*:=\bigcup_{i<4}h^*_{\beta,\beta_i}(w_i)$.
Let $\bar x$ be in $\poss(\bar q,\bar w,\zeta+1)$.
\begin{enumerate}
    \item\label{item:wegge1} $\bar q$ fits $(\bar w,\zeta)$ iff $q^*$ fits $(w^*,\zeta)$.
    \item\label{item:wegge3} 
    $\bar r\le^+_{\bar w,\zeta}\bar q$ iff $r^*\le^+_{w^*,\zeta} q^*$.
  \item\label{item:wegge2} 
  Assume $\bar q$ fits $(\bar w,\zeta)$. Then $\bar q$ is $(\bar w,\zeta)$-canonical iff $q^*$ is $(w^*,\zeta)$-canonical.
    \item\label{item:wegge4} Assume that $\bar q$
  is $(\bar w,\zeta)$-canonical.
      Let $x^*$
be the union of the $h^*_{\beta,\beta_i}(x_i)$.
Then  $x^*\in \poss(q^*, w^*,\zeta+1)$; and every element of
$\poss(q^*, w^*,\zeta+1)$ is such an 
$x^*$
for some $\bar x\in \poss(\bar q,\bar w,\zeta+1)$.
\item\label{item:wegge5}
  Assume that $\bar q$
  is $(\bar w,\zeta)$-canonical.
  Then $\bar q$
  $(\bar w,\zeta+1)$-decides $(\ntau_i)_{i<4}$ iff $q^*$  $(w^*,\zeta+1)$-decides
   all $h^*_{\beta,\beta_i}(\ntau_i)$.
\end{enumerate}
\end{lemma}  
\begin{proof}
Assume $\alpha\in w_i$. Set $\alpha':= h^*_{\beta,\beta_i}(\alpha)\in w^*$  and $q':=h^*_{\beta,\beta_i}(q_i)$.

(\ref{item:wegge1}) 
    Assume 
    $q_i,\alpha$
    satisfy $q_i\restriction \alpha\Vdash \zeta\in C^{q_i(\alpha)}$.
    By absoluteness they satisfy it in $M_{\beta}$,
    so the $h^*_{\beta,\beta_i}$-images
    $q',\alpha'$ satisfy it in 
    $M_{\beta_i}$, which again is absolute; and 
    $q^*\restriction\alpha'\le q'\restriction \alpha'$
    forces that
    $q^*(\alpha')=q'(\alpha')$.
    For the other direction, assume (in $M_{\beta}$) some $s\le q_i\restriction \alpha$
    forces $\zeta\notin C^{q_i(\alpha)}$.
    Then $h^*_{\beta,\beta_i}(s)$ 
    is compatible with $q^*$ and forces $\zeta\notin C^{q'_i(\alpha')}
    =C^{q^*(\alpha')}$.

   In the same way we can show (\ref{item:wegge3}),
   as well as (\ref{item:wegge5}) and the trivial directions of 
   (\ref{item:wegge2}), (\ref{item:wegge4}).
   E.g.,
   if $\bar q$ is $(\bar w,\zeta)$-canonical,
   then $q^*$ is $(w^*,\zeta)$-canonical.
   For this, use the fact
   that 
   every element $y^*\in \poss(q^*, w^*,\zeta+1)$
   ``induces'' a coherent possibility $\bar y$
   (which is true whether $\bar q$ is canonical or not). And if additionally
   $\bar x\in \poss(\bar q, \bar w,\zeta+1)$,
   then $x^*\in\poss(q^*, w^*,\zeta+1)$;
   and if each $q_i$ forces that $x_i\triangleleft G$
   implies $\ntau_i=x^i$, then 
   $q^*$ forces that $x^*\triangleleft G$ implies 
   $h^*_{\beta,\beta_1}(\ntau_i)=
   h^*_{\beta,\beta_1}(x^i)$.
   
   We omit the (also straightforward) 
   proofs of the other directions 
   of (\ref{item:wegge2}) and (\ref{item:wegge4})
   (which we do not need in this paper).
\end{proof}

In the following, whenever we mention $q^*$ or $w^*$, we assume $\bar w$,  $\bar q$ to be coherent and in $M_{\beta}$.
We will (and can) use $x^*$ only if $\bar q$ additionally is canonical
(otherwise $x^*$ will generally not be a possibility for $q^*$).
In this case, every $P$-generic filter containing $q^*$
will select an $x^*$ for some coherent possibility $\bar x$.

\begin{lemma}\label{lem:noitsnot}
  Assume $\bar q$ is coherent,
  $\nsigma_i$ are $P$-names in $M_\beta$
  for  elements of $2^\lambda$, and\footnote{As usual, $V_{\beta+1}$ denoted the $P_{\beta+1}$-extension.}
  $q_0\Vdash \nsigma_0\notin V_{\beta+1}$.
  Then there is a coherent $\bar r\le \bar q$,
  and sequences $(\zeta^j)_{j\in\lambda}$
  and $(\bar w^j)_{j\in\lambda}$
  such that 
  $\bar r$ is $(\bar w^j,\zeta^j)$-canonical
  for all $j$,
  and for all $\bar x\in \poss(\bar r, \bar w^j,\zeta^j+1)$
  there is some $\ell\in I^*({>}\zeta^j,{<}\zeta^{j+1})$
  and $\bar b=(b_i)_{i<4}$, with $b_i\in 2$, violating majority\footnote{I.e., 
  $b_0=1-\majority_{i=1,2,3}(b_i)$.} such that 
  for all $i<4$
  \[
      r_i\Vdash x_i\triangleleft G \rightarrow 
      \nsigma_i(\ell) = b_i.
   \]
\end{lemma}

As the $p_{\beta_i}$ are coherent, 
we can apply the lemma 
to 
$\nsigma_i:=\na_{\beta}$ (for all $i$) and get:
\begin{corollary}\label{cor:qwghqr}
If $p_{\beta}\Vdash \na_{\beta}\notin V_{\beta+1}$,
then there is a coherent $r^*\le \bigwedge_{i<4}p_{\beta_i}$ forcing that
\[
    \lnot\, \bigl(\na_{\beta_0}=^*\majority_{i=1,2,3}(\na_{\beta_i})\bigr).
\]
\end{corollary}


\begin{proof}[Proof of the lemma]
We will construct (in $M_\beta$), by induction on $j\in\lambda$,  $\zeta^j$, $\bar w^j$ and  $\bar r^j$
with $r^0_i=q_i$,
such that the following holds:
\begin{enumerate}
    \item $\bar r^j$ is coherent.
    \item $\bar w^j$ is coherent,
    for each $i<4$
    the $w_i^j$ are increasing with $j$, and their union
    covers $\bigcup_{j\in\lambda}\dom(r_i^j)$.
    \item $\bar r^j$ is $(\bar w^j,\zeta^j)$-canonical.
    \item 
    $\bar r^k\le^{+}_{\bar w^j,\zeta^j} \bar r^j$ for $j<k$.
    \item\label{item:abortabort} 
    If  $\bar x\in \poss(\bar r^j,\bar w^j,\zeta^j+1)$,
    then there is an $\ell\in I^*({>}\zeta^j,{<}\zeta^{j+1})$
    and a $b\in 2$ such that for at least two $i_1,i_2$ in $\{1,2,3\}$,
    $r^{j+1}_i$ forces that
    $x_i\triangleleft G$ implies
    \begin{equation}\tag{$*$}\label{eq:hqwet}
    \nsigma_0(\ell)=1-b,
    \quad
    \nsigma_{i_1}(\ell)=b,
    \quad
    \nsigma_{i_2}(\ell)=b.
    \end{equation}
\end{enumerate}

Then we take the usual fusion limits, as in Lemma~\ref{lem:many}(\ref{item:trivial2}),
and are done.






For limits $j$, 
let $\bar r'$ be a (coherent) limit of $(\bar r^{j'})_{j'<j}$,
and set $\zeta^*:=\sup_{j'<j}(\zeta^{j'})$
and $w^*_i:=\bigcup_{j'<j} w^{j'}_i$ for each $i<4$.
Note that $\bar r'$ fits $(\bar w^*,\zeta^*)$.
Then we can find coherent 
$\bar r^*\le^+_{\bar w^*,\zeta^*} \bar r'$
which is $(\bar w^*,\zeta^*)$-canonical, as in 
Lemma~\ref{lem:many}(\ref{item:trivial3}). 

In successor cases $j=j'+1$
set $(\bar r^*,\bar w^*,\zeta^*):=(\bar r^{j'},\bar w^{j'},\zeta^{j'})$.

In any case we want to construct $\bar r^j$, $\bar w^j$, and $\zeta^j$.

Enumerate $\poss(\bar r^*,\bar w^*,\zeta^*+1)$ as 
$(\bar x^k)_{k\in K}$.

We define $\bar s^k$ for $k\le K$, with $\bar s^0:= \bar r^*$
and, as usual, taking (coherent) limits at limits, such that:
\begin{itemize}
    \item $\bar s^k$ is coherent.
    \item $\bar s^\ell\le^+_{\bar w^*,\zeta^* } \bar s^k$ for $k<\ell<K$.
    (This implies that $\bar s^k$ is $(\bar w^*,\zeta^*)$-canonical.)
    \item There is a
    $\xi^k$
    and an $\ell\in I^*({>}\zeta^*,{<}\xi^k)$
    and a $b\in 2$ such that
    \begin{equation}\tag{$**$}\label{eq:hqwet2}
    s_0^{k+1}\Vdash x^k_0\triangleleft G\rightarrow \ntau_0(\ell)=1-b \quad
          \text{and} \quad (\exists^{\ge2} i\in\{1,2,3\})\  
          s_i^{k+1}\Vdash x^k_i\triangleleft G\rightarrow \ntau_i(\ell)=b.
      \end{equation}
\end{itemize}
Assume we can construct these $\bar s^k$, $\xi^k$ for all $k\in K$,
then let $\bar s^K$ be again  a (coherent) limit.
We set $w^{j}_i:=w^*_i\cup\{\alpha_j\}$ such that 
$\bar w^{j}$ is coherent (and such that, by bookkeeping, 
all elements of $\dom(p^j_i)$ will be eventually covered), and find some 
$\zeta^{j}>\sup_{k\in K}(\xi^k)$ and
$\bar r^{j}\le^+_{\bar w^*,\zeta^*} r^*$
which is $(\bar w^{j},\zeta^{j})$-canonical, again as in  Lemma~\ref{lem:many}(\ref{item:trivial3}). 
Then $\bar r^{j}$,  $\bar w^{j}$ and $\zeta^{j}$ are as required.

So it remains to construct, for $k\in K$,
$\bar s^{k+1}$ and $\xi^k$, which we will do in the
rest of the proof. Set $\bar s:= \bar s^k$, 
$\bar x:=\bar x^k$, $\bar w:=\bar w^*$ 
and $\zeta:=\zeta^*$.
Recall that $\bar s$ is $(\bar w, \zeta)$-canonical,
$\bar x\in\poss(\bar s,\bar w,\zeta)$,
and we are looking for $\bar s^{k+1}\le^+_{\bar w,\zeta}\bar s$
which satisfies~\eqref{eq:hqwet2} for $\bar x$.

Set $s'_i:=s_i\wedge x_i$.
It is enough to construct $t_i\le s'_i$
such that:
\begin{itemize}
    \item Both
$t_i\restriction \beta$ 
and $t_i(\beta)\restriction (\lambda\setminus \zeta+1)$ 
are independent of $i$.
   \item $t_0\Vdash \ntau_0(\ell)=1-b \quad
          \text{and} \quad (\exists^{\ge2} i\in\{1,2,3\})\  
          t_i\Vdash \ntau_i(\ell)=b$.
\end{itemize}
Then we can define $\bar s^{k+1}$ in the usual way:
$\dom(s_i^{k+1})=\dom(t_i)$ 
(and we can assume $\dom(s_i)=\dom(t_i)$,
by using trivial conditions).
For $\alpha\in\dom(t_i)$,
if $t_i\restriction \alpha\notin G_\alpha$
then set 
$s_i^{k+1}(\alpha)$ to be $s_i(\alpha)$, otherwise
$t_i(\alpha)\vee (s_i(\alpha)\restriction \zeta+1)$
if $\alpha\in w_i$ and $t_i(\alpha)$ otherwise.
The resulting 
$\bar s^{k+1}\le^+_{\bar w,\zeta} \bar s$ is coherent
and $s^{k+1}_i$ 
forces that $x_i\triangleleft G$ implies
$t_i\in G$.





We have to introduce more notation:
Fix $j\ne i$, 
and $a\le s'_j$ and
$b\le s_i'\restriction \beta+1$ (in $P_{\beta+1}$)
such that $b\restriction\beta\le a$ and
$b\restriction \beta$ forces that $b(\beta)$ is stronger than 
$a(\beta)$ above $\zeta$
(i.e., $b\restriction \beta\Vdash\, (\forall \xi>\zeta)\, b(\beta)(\xi)\supseteq a(\beta)(\xi)$).
Then we define $b^{[j]} \wedge a$ by
\[
   (b^{[j]} \wedge a)(\alpha)(\xi)=\begin{cases}
      b(\alpha)(\xi) & \text{if }\alpha<\beta,\\
      x_j(\beta)(\xi) & \text{if }\alpha=\beta\text{ and }\xi\le \zeta,\\
      b(\beta)(\xi) & \text{if }\alpha=\beta\text{ and }\xi> \zeta,\\
      a(\alpha)(\xi)& \text{otherwise.}
   \end{cases}
 \]
Note that $b^{[j]} \wedge a$ is stronger than $a$, but generally not stronger than $b$.

By our assumption, $q_0$ and therefore $s'_0$
forces $\nsigma_{0}\notin V_{\beta+1}$.
So in an intermediate model $V[G_{\beta+1}]$,
there is some $\ell\in I^*({>}\zeta)$
such that 
$s'_0/G_{\beta+1}$ does not decide $\nsigma_0(\ell)$. 
Back in $V$, 
fix some $b_0\le s'_0\restriction (\beta+1)$ in $P_{\beta+1}$ 
which determines this $\ell$.

Find 
$r'_1\le b_0^{[1]}\wedge s'_1$ 
  which
  determines $\nsigma_1(\ell)$ to be $j_1$ for some $j_1\in 2$.
  Find 
  $r'_2\le (r'_1\restriction \beta+1)^{[2]}\wedge s'_2$ which determines $\nsigma_2(\ell)$ to be $j_2$;
  analogously  find 
  $r'_3\le (r'_2\restriction \beta+1)^{[3]}\wedge s'_3$ which determines $\name a(\ell)$ to be some $j_3$.
  Let $j\in 2$ be equal to at least two of $j_1,j_2,j_3$.
  
  Set $p:= (r'_3\restriction \beta+1)^{[0]}\wedge s'_0$.
  In any $P_{\beta+1}$-extension honoring $p\restriction\beta+1$,
  $\name \nsigma_0(\ell)$ is not determined by $p/G_{\beta+1}$, i.e.,
  there is an $t_0\le p$ 
  forcing that $\name a(\ell)=1-j$.
  
  We now set and $t_i:= (t_0\restriction \beta+1)^{[i]}\wedge r'_i$ for $i=1,2,3$.
  Note that $t_i\le r'_i\le s'_i$ extends $x_i$ and
  forces $\nsigma_i(\ell)$ to be $1-j$ if $i=0$
  and to be $j$ for at least two $i$ in $\{1,2,3\}$.
  \end{proof}
  
  

We can now easily show:
\begin{lemma}\label{lem:wrtqewqw}
  For all but non-stationary many $\beta\in S^{\mu}_{\lambda^+}$
  \begin{equation*}
         p_*\ \Vdash\ \na_\beta\in V_{\beta+1}
  \end{equation*}
\end{lemma}  

\begin{proof}
We started in this section with an arbitrary 
$\Delta$-system 
and showed that 
Corollary~\ref{cor:qwghqr} 
and Lemma~\ref{lem:itismajafterall} holds for this
system.

We now use a specific $\Delta$-system:

Assume towards a 
a contradiction that on a non-stationary set $S'$ there are $p_\beta\le p_*$
forcing $\na_\beta\notin V_{\beta+1}$. By strengthening we can assume that
$p_\beta$ canonically reads $\na_\beta$.
Let $M_\beta$ contain $p_\beta$ and let $S\subseteq S'$ be such that 
 $(M_\beta,p_\beta)_{\beta\in S}$ is a $\Delta$-system.
Fix $\beta_0<\beta_1<\beta_2<\beta_3$ in $S$.
By Corollary~\ref{cor:qwghqr} we get 
a coherent $\bar r$ stronger than 
$\bar p$ such that $r^*\Vdash
    \lnot\, \bigl(\na_{\beta_0}=^*\majority_{i=1,2,3}(\na_{\beta_i})\bigr)$.
This contradicts Lemma~\ref{lem:itismajafterall}.
\end{proof}




\subsection{Fixing the \texorpdfstring{$\mathbf\Delta$}{Δ}-system}\label{sec:fixit}

We now know that there is a stationary set
$S^0\subseteq S^\mu_{\lambda^+}$
such that for all $\beta\in S^0$,
$\na_\beta$
is forced (by $p_*$) to be in $V_{\beta+1}$
but not in $V_{\beta}$ (see Lemmas~\ref{lem:notinbeta} and~\ref{lem:wrtqewqw}).

For each $\beta\in S^0$ there is a 
$p'_\beta\le p_*$ in $P$ forcing that
$\na_\beta$ is equal to some 
$P_{\beta+1}$-name, call it
$\na^*_\beta$, and we choose $p_\beta\le p'_\beta$ (we only have to strengthen the part below $\beta+1$) 
which canonically reads $\na^*_\beta$.\footnote{So $p_\beta\restriction\beta+1$ reads $\na^*_\beta$,
but generally the whole $p_\beta$ may be required 
to force $\na_\beta=\na^*_\beta$.}

We now fix, as usual, for each $\beta\in S^0$, some elementary model $M_\beta$ containing $p_\beta$,
and fix $S\subseteq S^0$ such that 
$(M_\beta,p_\beta)_{\beta\in S}$ is a $\Delta$-system.

So $p_{**}:=p_\beta\restriction\beta\le p_*$ is 
independent of $\beta\in S$
(it is a $P_\alpha$-condition for some $\alpha\in\Delta$, independent of $\beta\in S$ ); and
$\na^*_\beta$ is read continuously by $p_\beta\restriction\beta+1$ via $(w'_\zeta)_{\zeta\in E'}$ for some 
$E'\subseteq \lambda$ club,
with $w'_\zeta\subseteq \beta+1$.
As usual, due to homogeneity $E'$ is independent of $\beta\in S$, 
and the $w'_\zeta$ 
are independent of $\beta$
apart from the shifting of the final coordinate $\beta$
via the mapping $h^*_{\beta_0,\beta_1}$;
the same holds for the decision functions that map
$\poss(p_\zeta, w'_\zeta,\zeta+1)$ to
$\na_\beta\restriction I^*({<}\zeta+1)$



Let $E$ be the limit points of $E'$, and set $w_\zeta:=\bigcup_{\nu<\zeta}w'_\nu$.
Then
   $\na_\beta\restriction I^*({<}\xi)$
    is 
   $(w_\xi,\xi)$-determined by $p_\beta$
   for all $\xi\in E$.



In the $P_\beta$-extension, only $\neta_\beta$ remains undetermined, i.e.,
there 
are $f_{\xi}$ for $\xi\in E$ such that
$p_\beta/G_\beta$ forces
$\na_\beta\restriction I^*({<}\xi)=f_\xi(\neta_\beta\restriction I^*({<}\xi))$.
The $f_\xi$ are canonically read 
from $p_\beta\restriction\beta$ in a way
independent of $\beta$ (due to homegeneity).

Recall that $x\in\poss(\tilde p,\xi)$
is equivalent to: $x\in 2^{I^*({<}\xi)}$ and $x$ extends $\eta^{\tilde p}\restriction I^*({<}\xi)$.
So the domain of $f_\xi$ is $\poss(\tilde p,\xi)$.



To summarize: 
\begin{fact}\label{fact:summary}
$(M_\beta,p_\beta)_{\beta\in S}$ satisfies:
\begin{itemize}
    \item $p_\beta\restriction\beta=:p_{**}\le p_*$ is 
    a $P_{\sup(\Delta)}$-condition independent of $\beta\in S$.
    \item $p_\beta(\beta)=:\tilde p$ is 
    a $P_{\sup(\Delta)}$-name independent of $\beta\in S$,
    \item There is a club-set $E\subseteq \lambda$ and, for $\xi\in E$,
    $P_{\sup(\Delta)}$-names $\name f_\xi:\poss(\tilde p,\xi)\to 2^{I^*({<}\xi)}$ 
    such that
    for all $\beta\in S$ and $\xi\in E$
    \[
    p_\beta \Vdash \na_\beta\restriction I^*({<}\xi)=\name f_\xi(\neta_\beta\restriction I^*({<}\xi)).
    \]
    \item 
    If $\beta\in S$,
    $\name x\subseteq \lambda$ is a $P_\beta$-name,
    $q\le p_{**}$ in $P_{\beta}$ and 
    $q,\name x$ are in $M_\beta$, then we can find
    $\alpha\in \Delta$ and 
    $p'_{**}\le q$ in $P_\alpha$ 
    which continuously reads $\name x$, $\ntau(\name x)$
    and $\ntau^{-1}(\name x)$ independently\footnote{This means:
$p'_{**}\in M_\gamma$ for all $\gamma\in S$, and there is a
way (independent of $\gamma\in S$) to continuously 
read $\name y_1,\name y_2,\name y_3$ modulo 
$p'_{**}$ from 
the generics below $\alpha$, and for all $\gamma\in S$
we have that $p'_{**}\wedge p_\gamma$ forces 
$\name y_1=\name x'$, $\name y_2=\ntau(\name x')$
and $\name y_3=\ntau^{-1}(\name x')$, where 
$\name x':=h^*_{\beta,\gamma}(\name x)$.} of $\beta$.
\end{itemize}
\end{fact}

The last item follows from Lemma~\ref{lem:ogottogott}; and
we will use it several times:
Before Corollary~\ref{cor:local} we find
$p^2_{**}\le p_{**}$ to get names for $U$, $F_\xi$ etc.\  that are independent of $\beta$; 
before Lemma~\ref{lem:Alarge}
we get $p^3_{**}\le p^2_{**}$
to get independent names for some unions, intersections and $\npi$-images;  
and finally after Corollary~\ref{cor:qwrqwr} 
we choose 
$q\le p^3_{**}$ to get an independent name for 
the generator $\Gen$.


\subsection{Local reading}\label{sec:localreading}

So we know that we can determine initial segments
of $\na_\beta$ 
from  initial segments of
$\neta_\beta$, more specifically, we can determine 
$\neta_\beta\restriction I$
from 
$\na_\beta\restriction I$ 
for $I:=I^*({<}\xi)$. 

In this section we show that 
on unboundedly many disjoint intervals of the form
$A:=I^*(\ge\xi,<\nu)$,
we can read $\na_\beta\restriction A$
from just $\neta_\beta\restriction A$ (without having to 
use the $\neta_\beta$-values below $A$).

The following definition (the notion of candidate) 
is only used in this section.
In the rest of the paper
we only need Corollary~\ref{cor:local}.

In the following, we work in $V_\beta$, the $P_\beta$-extension
$V[G_\beta]$ where we assume $\beta\in S$ and
$p_{**}\in G_\beta$.

\newcommand{\cand}{\mathrm{cand}}
\begin{definition} (In $V_{\beta}$)
  \begin{itemize}
  \item
     For $A\subseteq \lambda$ and $\bar x=(x_i)_{i<4}$, $x_i:A\to 2$,
     we say $\bar x$ honors majority above $\zeta$, if 
      \[
      x_0(\ell)=\majority_{i=1,2,3}x_i(\ell)\text{ for all }
      \ell\in A\cap I^*({\ge}\zeta).
      \]
      We say $\bar x$ honors $\tilde p$, if each $x_i$ is compatible
       with $\eta^{\tilde p}$ (as partial functions).
      \item 
        $\bar x=(x_i)_{i<4}$ is 
      a $(\zeta_0,\zeta_1)$-candidate,
      (for $\zeta_0\le \zeta_1$ both in $E$)
      if the $x_i\in\poss(\tilde p,\zeta_1)$ 
      honor majority above $\zeta_0$.
      
      (As elements of
      $\poss(\tilde p,\zeta_1)$ they automatically honor $\tilde p$.)
      \item
      If $\bar x$ is a $(\zeta_0,\zeta_1)$-candidate,
      we say ``$\bar y$ extends $\bar x$'' if
      $\bar y$ is a $(\zeta_1,\zeta_2)$-candidate\footnote{or equivalently, a 
      $(\zeta_0,\zeta_2)$-candidate} for some $\zeta_2\ge \zeta_1$
      and each $y_i$ extends $x_i$.
      
      Equivalently, $\bar y=\bar x^\frown \bar b$
      for some $\bar b$, with 
      $b_i:I^*({\ge}\zeta_1,{<}\zeta_2)\to 2$, which honors both majority
      and $\tilde p$.
      \item 
      A $(\zeta_0,\zeta_1)$-candidate  $\bar y$ is
      ``good'',  if for every candidate $\bar z$ of height $\xi>\zeta_1$
      that extends $\bar y$ we have:
\begin{equation}\label{eq:iub253}\tag{$*_1$}
f_\xi(z_0)(\ell) =\majority_{i=1,2,3}f_\xi(z_i)(\ell)\text{ for all }\ell\in I^*({\ge}\zeta_1,{<}\xi).
\end{equation}

  \end{itemize}
\end{definition}

\begin{plemma}\label{plem:qwtqw}
  (In $V_\beta$.)
  Every candidate can be extended to a good candidate.
\end{plemma}

\begin{proof}

  
  Assume otherwise, i.e., there is a $(\zeta',\zeta_0)$-candidate $\bar x$
  which is a counterexample,
      which means: 
  \begin{equation}\label{eq:wtqqwr22}\tag{$*_2$}
      \parbox{0.8\textwidth}{
      Whenever $\bar y$  is a $(\zeta_0,\zeta_1)$-candidate extending $\bar x$
  then there is a $\xi>\zeta_1$ and a $(\zeta_1,\xi)$-candidate $\bar z$ extending $\bar y$
  which violates~\eqref{eq:iub253}.}
  \end{equation}

 
  
  
  We now construct 
  $r_0\le \tilde p$ and, 
  for $i=1,2,3$,
  $Q_\beta$-names $r_i\le \tilde p$.
  All these conditions live on the same $C^*\subseteq E$
  with $\min(C^*)=\zeta_0$.
  The trunk of $r_i$ is $x_i$.
  

  We now construct inductively $C^*\restriction\zeta $ and 
  $r_i\restriction \zeta$.
  
  Assume we have determined that 
  $\zeta\in C^*$ and we have constructed each $r_i$ below $\zeta$.
  Set $r_0(\zeta):=\tilde p(\zeta)$ and pick
  $r_i(\zeta)$ 
  as in~\eqref{eq:rtqw}, i.e., they have 
  majority $\neta_\beta$
  and leave enough freedom to form a valid condition.
  
  We will now construct the $C^*$-successor $\xi$ of $\zeta$,
  together with $r_i$ on $I^*({>}\zeta,{<}\xi)$.
  
  Enumerate  
  all $(\zeta_0,\zeta+1)$-candidates extending $\bar x$
  as $(\bar y^k)_{k\in K}$.
  
  
  Let $\bar a^0$ be the empty $4$-tuple and set $\xi_0:=\zeta+1$.
  We will construct, for $k\in K$, $\xi_k$ and 
  some $\bar a^k$ 
  that honors majority and $\tilde p$, where $a^k_i$ has domain $I^*({\ge}\zeta+1,{<}\xi_k)$ 
  and extends $a^j_i$ if $j<k$.
  
  If $k$ is a limit, let $\bar a^x$ be the (pointwise) 
  union of $\bar a^j$ with $j<k$,
  and set $\xi_k:=\sup_{j<k}(\xi_j)$.
  
  Assume we already have $\bar a^j$.
  Extend ${{\bar y}^j}^\frown \bar a^j$
  to some candidate  ${{\bar y}^j}^\frown \bar a^{j+1}$ of some height $\xi_{j+1}$ in $E$
  such that
  \begin{equation}\label{eq:uhqwefg973t}\tag{$*_3$}
       \text{${\bar{y}^j}^\frown \bar a^{j+1}$ violates~\eqref{eq:iub253} for some 
       $\ell\in I^*({\ge}\xi_j,{<}\xi_{j+1})$.}
 \end{equation}
 We can due that due to~\eqref{eq:wtqqwr22}.
 
  So in the end we get some $\xi>\zeta$ in $E$ and $\bar b^\zeta$
  with domain $I^*({>}\zeta,{<}\xi)$ honoring majority and $\tilde p$
  such that
  \begin{equation}\label{eq:whetqr3235}\tag{$*_4$}
      \parbox{0.9\textwidth}{for every $(\zeta_0,\zeta+1)$-candidate $\bar y$
      extending $\bar x$,
      $\bar y^\frown \bar b^\zeta$ is a $(\zeta_0,\xi)$-candidate violating~\eqref{eq:iub253} for some $\ell\in I^*({>}\zeta,{<}\xi)$.}
  \end{equation}
  We then define $C^*$ below $\xi+1$ by adding only $\xi$,
  i.e., $\xi$ is the $C^*$-successor of $\zeta$.
  %
  We extend the conditions $r_i$ by $b^\zeta_i$ for $i<4$.
  I.e., we
  have $\eta^{r_i}(\ell)=b^\zeta_i(\ell)$.
    This ends the construction of  $r_i\le \tilde p$.
    
  \medskip


  Back in $V$, 
  assume that~\eqref{eq:wtqqwr22} is forced by some $q'\le p_\beta\restriction\beta$.
  Pick an increasing sequence $\beta_i$ ($i<4$) in $S$.
%
%
  We take the union of $q'$ and the $p_{\beta_i}$, call it $s$, and strengthen
  $s(\beta_i)=\tilde p$
  to $r_i$. The resulting condition $s'$ forces the following:
  

  \begin{itemize}
      \item $\na_{\beta_i}\restriction I^*({<}\xi)=f_\xi(\neta_{\beta_i}\restriction I^*({<}\xi))$  for all $\xi\in C^*$. This is because $s'\le p_{\beta_i}$, cf.\ Fact~\ref{fact:summary}.
      \item The $\neta_{\beta_i}$
      honor majority above 
      $\zeta_0$.
      This is because 
      for all $\zeta\in C^*$, the
      $r_i(\zeta)$ are chosen as in~\eqref{eq:rtqw} and therefore
      honor majority; 
      and for $\zeta\in \lambda\setminus (C^*\cup\zeta_0)$ we use values 
      $\bar b$ which honor majority.
      \item Accordingly, the 
      $\na_{\beta_i}$ honor majority above some $\gamma<\lambda$, cf.\ Lemma~\ref{lem:blal0}(\ref{item:bla101}).
      Pick $\zeta_1$ such that $\sup(I^*({<}\zeta_1))>\gamma$.
      \item So for all $\xi>\zeta_1$ 
      the $f_\xi(\neta_{\beta_i}\restriction I^*({<}\xi))$ honor majority above $\zeta_1$.
      \item
      Pick some $\zeta>\zeta_0,\zeta_1$ in $C^*$ with
      $C^*$-successor $\xi$.
      By construction of the $r_i$, 
      $\neta_{\beta_i}\restriction I^*({\ge}\zeta+1,<\xi)$
      is $b^\zeta_i$.
      As $r_i$ extends $x_i$, 
      $\bar y:=\neta_{\beta_i}\restriction I^*({<}\zeta+1)$
      is a $(\zeta_0,\zeta+1)$-candidate extending $\bar x$.
      So by~\eqref{eq:whetqr3235},
      the $\neta_{\beta_i}\restriction I^*({<}\xi)$
      violate~\eqref{eq:iub253} at some $\ell\in I^*({>\zeta},{<\xi})$, a contradiction.
      \qedhere
  \end{itemize}  
\end{proof}

   Let $U\subseteq \lambda$ be club.
   Set $U^\odd$ to be the odd elements\footnote{I.e., if $(u_\alpha)_{\alpha<\lambda}$ is the canonical enumeration
   of $U$, then $\zeta\in U$ is in $U^\odd$ if $\zeta=u_{\delta+2n+1}$ for $\delta$
   a limit (or 0) and $n\in\omega$.} of $U$.
   For $\xi\in U^\odd$ with $U$-successor $\nu$, set
 \[
   A^U_\xi:=I^*({\ge}\xi,{<}\nu)
 \]  

\begin{lemma}\label{lem:local}
  (In $V_{\beta}$.)
  There is an $r_0\le \tilde p$,
  a club $U\subseteq C^{r_0}\subseteq E$
  and, for $\xi\in U^\odd$, an
  $F_\xi:2^{A^U_\xi}\to 2^{A^U_\xi}$
  such that
  \begin{itemize}
      \item $r_0\wedge p_\beta/G_\beta$ forces that 
      $F_\xi(\neta_\beta\restriction A^U_\xi)=
      \na_\beta\restriction A^U_\xi
      $.
      \item $F_\xi$ is not constant: There are, for $k=0,1$,
      $z_\xi^k$ in $\poss(r_0,I^*({<}\nu))$
      and $\ell_\xi\in A^U_\xi$ such that 
      $F_\xi(z_\xi^k\restriction A^U_\xi)(\ell_\xi)=k$.
      (Again, $\nu$ is the $U$-successor of $\xi$.)
  \end{itemize}
\end{lemma}
(Note: Only those elements of $2^{A^U_\xi}$ that are compatible with $r_0$
are relevant as arguments for $F_\xi$.)

\begin{proof}
  We construct $r_i$ for $i<4$ and $U$ iteratively;
  $C^{r_i}$ will be independent of $i$, call it $C$.
  
  All $r_i$ have the same trunk as $\tilde p$; 
  i.e., $\min(C)=\min(C^{\tilde p})=:\zeta_0$ and
  $r_i\restriction\zeta_0:=\tilde p\restriction\zeta_0$.
  We also set $\min(U)=\zeta_0$.

  For all $\zeta\in C$, we choose 
  some $r^*_i(\zeta)$ 
  as in~\eqref{eq:rtqw}, i.e., $r^*_0(\zeta)=\tilde p(\zeta)$,
  and the $r^*_i(\zeta)$ for $i=1,2,3$ 
  are such that the majority of their
  generics would be the $r^*_0(\zeta)$-generic.

  Assume that we already know that some $\zeta$ is in $U$ (which is a subset of $C$), and 
  that we know $r_i\restriction \zeta$ for $i<4$.

  We now construct the $U$-successor $\xi$ of $\zeta$,
  $C\restriction [\zeta,\xi]$, and $r_i(\nu)$
  for $i<4$ and $\nu\in[\zeta,\xi)$.
  
  \begin{itemize}
      \item 
      Even case:
  If $\zeta$ is an even element of $U$, we 
  start with $r_i(\zeta):=r_i^*(\zeta)$, but then
  add a ``shield'', or
  ``isolator'' above $\zeta$: 
  As in the previous proof, we iterate over all $\zeta+1$-candidates $\bar y^j$, but 
  but in~\eqref{eq:uhqwefg973t}, instead of 
   violating~\eqref{eq:iub253} for some $\ell$, we demand that 
   ${\bar y^j}^\frown \bar z^{j+1}$ is good.
  (We already know that every candidate can be extended to a good one.)
  Accordingly, 
we get some $\xi>\zeta$ and $\bar b^\zeta$
with domain
$I^*({>}\zeta, {<}\xi)$ 
(and honoring majority and $\tilde p$)
such that 
$\bar y^\frown \bar b^\zeta$ is good
for every 
candidate $\bar y$ of height $\zeta+1$; i.e.:
\begin{equation}\label{eq:iub253b}\tag{$*_4'$}
\parbox{0.8\textwidth}{If $\bar z$ is a $(\zeta+1,\nu)$-candidate whose restriction to $I^*({>}\zeta, {<}\xi)$ is 
$\bar b^\zeta$, then the $f_{\nu}(z_i)$ honor majority above $\xi$.}
\end{equation}
We now let this $\xi$ be the successor of $\zeta$ in both $C$ and $U$ (and extend each $p_i(\zeta)$ by $b_i$).

  \item {Odd case:} Now assume $\zeta$ is odd in $U$.  
  Then we choose 
  some $\xi>\zeta$ in $C^{\tilde p}$ large enough such that 
  there are,
  for $k=0,1$,
  $z^k_\xi$ in $\poss(\tilde p,\xi)$ compatible with
  all the $r_0$ constructed so far,
  such that the $f_{\xi}(z^k_\xi)(\ell)=k$ for some
  $\ell> I^*({<}\zeta)$.
  (Such $\xi$ and $\ell$ have to exist as $\na_\beta$ is not 
  in $V_\beta$.)

  We let $C$ restricted to $[\zeta,\xi]$ be the same as $C^{\tilde p}$, and set 
  $r_i(\nu):=r_i^*(\nu)$ for $\nu\in C\cap [\zeta,\xi)$.
  (For $\zeta\in [\zeta,\xi)\setminus C$ there is no
  freedom left, i.e., 
  $\tilde p(\zeta)$ is already completely determined,
  so the only choice for any $r\le \tilde p$ is 
  $r(\zeta)=\tilde p(\zeta)$.)
  \end{itemize}
  This ends the construction of $U$ and of $r_i$ (for $i<4$).

  \medskip 
  
  Pick $\xi\in U^\odd$, let $\zeta$ be the 
  $U$-predecessor and $\nu$ the $U$-successor.
  We have to show that we can determine (modulo $p_{\beta}$) 
  $\na_{\beta}\restriction I^*(\ge\xi,<\nu)$
  from $\neta_{\beta} \restriction I^*(\ge\xi,<\nu)$ alone. (We already know that we can determine it from
  $\neta_{\beta} \restriction I^*(<\nu)$.)
  
  Fix any $z^\zeta_*\in\poss(r_0,\zeta+1)$.
  Let $x_0\in\poss(r_0,\nu)$. In particular
  $x_0$ extends $b_0^\zeta$.
  For $i=1,2,3$, let $x_i$ be the copy of 
  $x_0$ with
  the initial segment $x_0\restriction \xi$
  replaced by 
  $z^\zeta_*{}^\frown b_i^\zeta$.
  Note that $\bar x$ is a candidate extending $\bar b^\zeta$.
  Accordingly the $f_\nu(x_i)$ honor majority above $\xi$.
  So we can define 
  \[
     F_{\xi}(x_0\restriction A^U_\xi):=\majority_{i=1,2,3}f_{\nu}(x_i)
     \restriction A^U_\xi = f_\nu(x_0) \restriction A^U_\xi.
  \]
  This is well-defined,\footnote{Assume $y$ and $x$ 
  in $\poss(r_0,\nu)$
  are the identical
  restricted to $A^U_\xi$. 
  Then $y$ defines the same $(x_i)_{i=1,2,3}$ and thus the 
  same $F_\xi$.} and $r_0\wedge p_\beta/G_\beta$ forces that
  $F_\xi(x_0\restriction A^U_\xi)=\na_\beta\restriction A^U_\xi
  $.
\end{proof}

We now summarize this lemma, which was shown in
$V_\beta$ for some $\beta\in S$, from the point of view of the ground model.
The lemma only uses the parameters
$\neta_\beta$ and $\na_\beta$ (and $\tilde p$, which
is just $\neta_\beta(\beta)$), so by absoluteness
$M_\beta$ knows that the Lemma is forced by $p_{**}$.
Accordingly, we can find $P_\beta$-names for $U$, $F_\xi$ etc
in $M_\beta$.
Using the last item of Fact~\ref{fact:summary},
we can strengthen $p_{**}$ to 
$p_{**}^2$ to canonically read these names:

\begin{corollary}\label{cor:local}
    There is
        an $\alpha\in \Delta$, a $p_{**}^2\le p_{**}$ in $P_\alpha$ and 
        $P_{\alpha}$-names
    for: A condition $r_0\le \tilde p$, a set $U$
    and a sequence $(F_\xi, z^0_\xi, z^1_\xi,\ell^0_\xi, \ell^1_\xi)_{\xi\in U}$,
    such that the following holds for all
    $\beta\in S$, where we set
     \[
            p^+_\beta\text{ to be the condition }p^2_{**} \wedge p_\beta\text{ where we strengthen } p_\beta(\beta)\text{ to } r_0.
        \]
    \begin{enumerate}
        \item $\alpha$, the condition $p^2_{**}$ and all the names are in $M_\beta$.
        \item $p^2_{**}\Vdash  U\subseteq C^{r_0}\subseteq \lambda$ club.
        \item  for $k=0,1$: 
    $p^2_{**}\Vdash\ \forall \xi\in U^\odd\,  \bigg(z_\xi^k\in \poss(r_0,I^*({<}\nu))
    \,\&\,
      \ell_\xi\in A^U_\xi
      \,\&\,
      F_\xi(z_\xi^k\restriction A^U_\xi)(\ell_\xi)=k
      \bigg)$.
        \item $p_\beta^+\Vdash (\forall \xi\in U^\odd)\, 
            F_\xi(\neta_\beta\restriction A^U_\xi)=
      \na_\beta\restriction A_\xi$, where we define 
      \[
      A_\xi\text{ to be }I^*({\ge}\xi,{<}\nu)\text{
    with $\nu$ the $U$-successor of $\xi$}.   
    \]
    \end{enumerate}

\end{corollary}

\subsection{Finding the generator}
In this section we use these
$p^2_{**}$, $r_0$, $(F_\xi, z^0_\xi, z^1_\xi,\ell^0_\xi, \ell^1_\xi)_{\xi\in U}$.

We start working in $V_{\beta}=V[G_\beta]$,
where we assume $p^2_{**}\in G_\beta$.

Let $\xi\in U^\odd$ and $\nu$ its $U$-successor.
Set
\begin{align}
\nonumber A_\xi&:=I^*({\ge}\xi,{<}\nu)
,&
A^?_\xi&:= A_\xi\setminus \dom(\eta^{r_0})
,\\\label{eq:defodd}
\odd&:=\bigcup_{\xi\in U^\odd} A_\xi
,&
\odd^?&:=\bigcup_{\xi\in U^\odd} A^?_\xi=\odd\setminus \dom(\eta^{r_0})
.
\end{align}

For $F_\xi$ it is enough to use 
$\neta_\beta\restriction A^?_\xi$ as input
(the part in $A_\xi\setminus A^?_\xi$ is determined
anyway by $r_0$), and \emph{every} element of $2^{A^?_\xi}$ is compatible with $r_0$ (and thus a possible input for $F_\xi$).
Identifying $2^B$ and $\mathcal P(B)$ as usual, we get:
\[
  F_\xi: \mathcal P(A^?_\xi)\to \mathcal P(A_\xi)
\]
is such that $p^+_\beta/G_\beta$ forces
\[
  F_\xi(\neta_\beta\cap A^?_\xi)
  = 
  \na_\beta\cap A_\xi,
\]
 We now define 
\[
    F:\mathcal P(\odd^?)\to \mathcal P(\odd)\quad \text{by}\quad x\mapsto \bigcup_{\xi\in U^\odd}F_\xi(x \cap A^?_\xi).
\]

So in particular $p^+_\beta/G_\beta$ forces that
\begin{equation}\label{eq:hourglass}
  F(\neta_\beta\cap \odd^?)= \na_\beta \cap \odd.  
\end{equation}

Note that for every $z\subseteq \odd^?$ (in $V_\beta$ that is)
there is an $r'\le r_0$
forcing that $\neta_\beta\cap \odd^?=z$.
($C':=U\setminus U^\odd$ is club, so it is enough to leave freedom at $C'$
and we may assign arbitrary values everywhere else.)


\medskip

Back in the ground model $V$, using the last item of Fact~\ref{fact:summary} again,
we can strengthen $p^2_{**}$ to 
$p^3_{**}$ so that 
\begin{equation}\label{eq:hhourglass}
\text{$p^3_{**}$ canonically reads each of the following (countably many) sets:}\footnote{We can do this for $\lambda$ many sets, 
of course; but we cannot assume e.g. that
$\npi(z)\in V_\beta$ for \emph{all} $z\in V_\beta$,
let alone that each such $\npi(z)$ is canonically read
by $p^3_{**}$.} 
\end{equation}
 

\begin{itemize}
    \item $(A_\xi)_{\xi\in U^\odd}$, $\odd$, 
$r_0$,
$(A^?_\xi)_{\xi\in U^\odd}$, $\odd^?$ (actually, these are already read by $r^2_{**}$).
    \item The closure of these sets under $\npi$, $\npi^{-1}$, finite unions,  and finite intersections.
\end{itemize}

In particular, the (names for) all these sets
are independent of $\beta\in S$,
modulo $p^3_{**}$.\footnote{But we need $p^{+}_\beta$
to force that these names have anything to do with $\na_\beta$.}

 

\begin{lemma}\label{lem:Alarge}
  (In $V$)
  $p^3_{**}\Vdash |\npi(\odd^?)\cap \odd|=\lambda$.
\end{lemma}
\begin{proof}
  Let $q\le p^3_{**}$ in $P_\beta$
  be arbitrary. We have to show that $q$
  does not force (in $P_\beta$) $|\npi(\odd^?)\cap \odd|<\lambda$.

  For $\xi\in U^\odd$ and $k=0,1$, 
  use $r_0$, $p_\beta^+$, 
  $z^k_\xi$ and $\ell_\xi$ as in
  Corollary~\ref{cor:local} and set
  $b^k_\xi:= z^k_\xi\cap A_\xi^?$.

  
  For $k=0,1$, set $B^k:=\bigcup_{\xi\in U^\odd}(b^k_\xi)$.
  Note that $F(B^1)\setminus F(B^0)$ contains
  $\{\ell_\xi:\, \xi\in U^\odd\}$, a set of size $\lambda$.
  
  Pick increasing $(\beta_i)_{i< 4}$ in $S$ with
  $\beta_0=\beta$. 
  Set  
  $s:=q\wedge \bigwedge_{i<4} p^+_{\beta_i}\in P$.

  Now for each $i<4$, strengthen $s(\beta_i)$ (i.e., $r_0$) 
  as follows: At the even intervals in some way 
  that together they honor majority;
  and at the odd intervals (where we do not have to leave freedom) 
  to the value $B^{\sign(i)}$
  (where $\sign(k)=0$ for $k=0$ and $1$ for $k=1,2,3$).
  
  Accordingly, we have
  \[\npi(\neta_{\beta_i})\cap \odd = F(\neta_{\beta_i}\cap \odd^?)=F(B^{\sign(i)}),\]
  or, when we split $\npi(\neta_{\beta_i})$
  into the parts in and out of $\npi(\odd^?)$:
  \[\biggl(\bigl( \npi(\neta_{\beta_i})\setminus \npi(\odd^?)\bigr) \cap \odd \biggr)
  \cup
  \biggl( \npi(\neta_{\beta_i})\cap \npi(\odd^?) \cap \odd \biggr) =^* F(B^{\sign(i)})
  \]

  Now assume towards a contradiction that 
  $\npi(\odd^?)\cap \odd=^*\emptyset$.
  Then we get:
  \begin{equation}\label{eq:qwijrwrq}
  \bigl( \npi(\neta_{\beta_i})\setminus \npi(\odd^?)\bigr) \cap \odd =^* F(B^{\sign(i)}).
  \end{equation}
  But on the other hand we have:
  \begin{gather*}
   \neta_{\beta_0}\setminus \odd^? = \majority_{i=1,2,3}(\neta_{\beta_i}\setminus \odd^?)
  \text{, so } 
  \\\nonumber
  \npi( \neta_{\beta_0})\setminus \npi(\odd^?) =^* 
  \npi\big( \neta_{\beta_0}\setminus \odd^?\big) = 
  \npi\big(\majority_{i=1,2,3}(\neta_{\beta_i}\setminus \odd^?)\big)=^*
  \\\nonumber
  =^*\majority_{i=1,2,3}\big(\npi(\neta_{\beta_i}\setminus \odd^?) \big)=^*
  \majority_{i=1,2,3}\big(\npi(\neta_{\beta_i}\big)\setminus \npi(\odd^?)
  \text{, and } 
  \\\nonumber
  \bigl(\npi(\neta_{\beta_0})\setminus \npi(\odd^?)\bigr)\cap \odd =^* 
  \majority_{i=1,2,3}\biggl(\bigl(\npi(\neta_{\beta_i})\setminus \npi(\odd^?)\bigr)\cap \odd\biggr).
  \end{gather*}
  Applying~\eqref{eq:qwijrwrq} to both sides of the last line, we get 
  $F(B^0)=^*\majority_{i=1,2,3}F(B^{\sign(i)})=F(B^1)$,\linebreak[1]
  a contradiction.
\end{proof}

Set\begin{equation}\label{eq:defX}
  \nX:=\odd^?\cap \npi^{-1}(\odd).  
\end{equation}
By choice of $p^3_{**}$,
$\nX$ and $\npi(\nX)$ are canonically read by $p^3_{**}$ 
(and independent of $\beta$).

We now show that $F(z)\cap \npi(\nX)=\npi(z)$
for $z\subseteq\nX$. Again, here we are talking about $z\in V_\beta$. To make that more explicit, let us 
formulate in the ground model $V$:
\begin{lemma}\label{lem:ipqjet}  For $\beta\in S$,
  \\\mbox{}\hfill
  $p^3_{**}\Vdash_{P_\beta}\ \bigg(\ 
  |\nX|=\lambda\text{, and 
  for all }z\subseteq \nX,\  
  p^{+}_\beta/G_\beta\Vdash 
      \npi(z)=^* F(z) \cap \npi(\nX)\bigg).$
\end{lemma}  
(Note that,
other than $F(z)$,
 $\npi(z)$ will generally not be in $V_\beta$, and we have to force with $p^{+}_\beta/G_\beta$.)


\begin{proof}
Work in $V_\beta$.
$|\nX|=\lambda$  follows from Lemma~\ref{lem:Alarge},
as $\npi(\nX)=^*\npi(\odd^?)\cap\odd$.
  
Set $y:=\neta_\beta\cap \odd^?$.
So by~\eqref{eq:hourglass}, $p^+_\beta/G_\beta\le r_0$ forces:
$F(y)= \npi(\neta_\beta)\cap\odd$.
As $\npi(\nX)\subseteq^* \odd$, we get
$F(y)\cap \npi(\nX)=^*\npi(\neta_\beta)\cap\npi(\nX)$.
Then $y\subseteq^* \npi^{-1}(\odd)$ (or equivalently, $y\subseteq^* \nX$)
implies $y=^*y\cap \npi^{-1}(\odd)=\neta_\beta\cap \nX$ and thus $\npi(y)=^*\npi(\neta_\beta)\cap \npi(\nX)$.
To summarize:
\begin{equation}\label{eq:wqehtewt}\tag{$*$}
    p^+_\beta/G_\beta\Vdash\bigg(
y\subseteq^* \nX\ \rightarrow\     
    \npi(y) =^* F(y)\cap \npi(\nX),
\text{ for }y:=\neta_\beta\cap \odd^?    
\bigg)
\end{equation}

  Now back in $V$ assume towards a contradiction that some $q\le p^+_\beta$ forces that 
  the lemma fails, i.e., that $\name z\subseteq \nX$
  in $V_\beta$ is a counterexample (in the final extension).
  By absoluteness, we can assume that $q$ and $\name z$
  are in $M_\beta$, in particular $\name z$ is a 
  $P_\beta$-name in $M_\beta$. Strengthen 
  $q\restriction \beta$ to canonically read $\name z$.
  So for every $\beta'\in S$, 
  $h^*_{\beta,\beta'}(\name z)$ will be evaluated 
  in $V_{\beta'}$ to
  the same $z\subseteq \lambda$ as $\name z$ in $V_\beta$.
   
  Chose a $\beta'$ above $\supp(q)$.  
  Then we can strengthen $q\wedge p_{\beta'}$ at
  index $\beta'$, i.e., $r_0$,  
  to some $r_1$ that forces
  $\neta_\beta\cap \odd^?=h^*_{\beta,\beta'}(\name z)$.
  (Recall that we can fix the values in the odd intervals, as
  the even intervals still form a club). 
  Let $G$ be $P$-generic containing $q\wedge p^+_{\beta'}\wedge r_1$. Then we have:
    \begin{itemize}
        \item The evaluation of
        $h^*_{\beta,\beta'}(\name z)$ in $V_{\beta'}$, 
        is the same as the evaluation of
        $\name z$ in $V_\beta$, call it $z$.
        \item Also the evaluation of $\nX$ and $F$ are 
        the same  $\beta$ and $\beta'$, cf.~\eqref{eq:hhourglass}.
        \item $z\subseteq \nX$ 
        is a counterexample (as this is forced by $q$).
        
        In particular, $z\subseteq \nX$ and $\npi(z)\ne^* F(z)\cap\npi(X)$
        in the final extension.
        
        \item $p_{\beta'}\wedge r_1$ forces
        in $V_{\beta'+1}$
        that $\neta_\beta\cap\odd^?=z$;
        also we have just seen that $z\subseteq \nX$; 
        and so $\npi(z)=^*F(z)\cap \npi(X)$ 
        by~\eqref{eq:wqehtewt},
        a contradiction.\qedhere
    \end{itemize}
\end{proof}

   For $\xi\in U^\odd$, we define the following $P_\beta$-names (independent of $\beta$):\footnote{
   More concretely, canonically read by 
   $p^3_{**}$, see~\eqref{eq:hhourglass}.} 
   \begin{align*}
   \nx_\xi&:=  A^?_\xi\cap\nX
   &
   \ny_\xi&:= A_\xi\cap \npi(\nX)\\
   \text{so}\quad\bigcup_{\xi\in U^\odd}\nx_\xi&=\nX
   & \bigcup_{\xi\in U^\odd} \ny_\xi&=
    \odd\cap \npi(\nX)=^* \npi(\nX),
   \end{align*}
   as well as\begin{center}
   \begin{tabular}{llcl}
    &$F'_\xi: \mathcal P(\nx_\xi)\to \mathcal P(\ny_\xi)$
   &by
   &$a\mapsto F_\xi (a)\cap \npi(\nX)$, \\[2ex]
   and
   & $F':P(\nX)\to P(\npi(\nX))$
   &by
   &$z\mapsto \bigcup_{\xi\in U^\odd}F'_\xi(z\restriction \nx_\xi)=F(z)\cap \npi(\nX).$
   \end{tabular}
   \end{center}
   So the $p^3_{**}$ forces that
   for all $z \in V_\beta$ the 
   following is forced by $p^{+}_\beta/G_\beta$:
   \begin{equation}
       z\subseteq \nX\ \rightarrow F'(z)=^*\npi(z),
       \quad 
       \text{in particular }F'(\nX)=^*\npi(\nX),
       \quad
       \text{also }F'(z)\subseteq \npi(\nX)\text{ for all }z
   \end{equation}

\begin{lemma}   
   $p^3_{**}$ forces: 
   For almost all $\xi\in U^\odd$, 
   $F'_\xi$ is a Boolean algebra isomorphism
   from $P(\nx_\xi)$ to $P(\ny_\xi)$.
\end{lemma}

\begin{proof}
      \emph {All and nothing:}
  We claim that for almost all $\zeta$,
  $F'_\zeta(\nx_\zeta)=\ny_\zeta$.
  Assume that $\ell\in \ny_\zeta\setminus F'_\zeta(\nx_\zeta)\subseteq 
  \npi(\nX)$. Then 
  $\ell\in \npi(\nX)$, and 
  $\ell$ is not in $F'(\nX)=^* \npi(\nX)$,
  so there cannot be many such $\ell$.
  Similarly 
  $F'_\zeta(\emptyset)=\emptyset$
  for almost all $\zeta$ .

  

  \emph{Unions:} We claim that for almost all $\zeta$,
  $F'_\zeta(a)\cup F'_\zeta(b)= 
  F'_\zeta(a\cup b)$ for all subsets $a,b$ of $\nx_\zeta$.
  Let $A\subseteq \lambda$ be the set of counterexamples,
  i.e., for $\xi\in A$
  there are 
  $\ell_\xi\in \ny_\xi$, and $a_\xi$, $b_\xi$ subsets
  of $\nx_\zeta$ such that 
  $\ell_\xi\in 
  \bigl(F'_\xi(a_\xi)\cup F'_\xi(b_\xi)\bigr)\Delta
  F'_\xi(a_\xi\cup b_\xi)$.
  Set $x:=\bigcup_{\xi\in A}a_\xi $
  and $y:=\bigcup_{\xi\in A}b_\xi $.
  Then $\ell_\xi$ is in $\bigr(F'(x)\cup F'(y)\bigl) \Delta F'(x\cup y)=^*\emptyset$, so $A$ cannot be large.


  
  
  
  \emph{Complements:}
  We claim that for almost all $\xi$,
  $F'_\xi(a)\cap F'_\xi(\nx_\xi\setminus a)=\emptyset$.
  Let $A$ be the set of counterexamples, i.e., for $\xi\in A$
  there is an $a_\xi\subseteq \nx_\xi$ and $\ell\in \ny_\xi$ 
  such that
  $\ell_\xi\in F'_\xi(a_\xi)\cap F'_\xi(\nx_\xi\setminus a_\xi)$. 
  Then $\ell_\xi$ is in $F'(\bigcup_{\zeta\in A}a_\zeta )
  \cap F'(\bigcup_{\zeta\in A}\nx_\xi\setminus a_\zeta)=^*\emptyset$, so $A$ cannot be large.
  
  
  
  \emph{Injectivity:}
  We already know that union and complements
  (and thus disjointness) are preserved,
  so it is enough to
  show that a nonempty set is mapped to a nonempty set.
  
  Assume this fails often, then we get an $x\subseteq \nX$
  of size $\lambda$ such that $\emptyset=F'(x)=^*\npi(x)$,
  a contradiction.
  
  \emph{Surjectivity:}
  Assume surjectivity fails often; 
  i.e., there are many $b_\zeta\subseteq \npi(\nX)\cap\odd$ not in the range of $F'_\zeta$. Let $y$ be the union
  of those $b_\zeta$.
  Pick $x\subseteq \lambda$ such that $\npi(x)=^*y\subseteq \npi(\nX)$.
  So we can assume $x\subseteq \nX$ and so
  $F'(x)=^*y$, which implies that 
  $F_\zeta(x\cap \nx_\zeta)=y\cap A_\zeta=b_\zeta$
  for almost all $\zeta$, a contradiction.
\end{proof}

\begin{lemma} 
For each $\beta\in S$:
$p^3_{**}$ forces (in $P_\beta$):
There is a $\Gen:\nX \to \npi(\nX)$ bijective
such that for all $z\subseteq\nX$ (in $V_\beta$),
$p^{+}_\beta/G_\beta$ forces
$\npi(z)=^*\Gen''z$.
\end{lemma}

\begin{proof}
Every Boolean algebra isomorphism from $P(A)$ to 
$P(B)$ is generated by 
a bijection from $A$ to $B$ (the restriction to
the atoms). 
So there is an $U'\subseteq U^\odd$
with $|U^\odd\setminus U'|<\lambda$
such that 
$\zeta\in U'$ implies that 
$F'_\zeta$ is generated by some bijection $g_\zeta: \nx_\zeta\to\ny_\zeta$.
So $F'$ is generated by 
   $g:=\bigcup_{\zeta\in U'} g_\zeta$;
   and 
   we can change $g$ into a bijection from $\nX$
   to $\npi(\nX)$ by changing less than $\lambda$
   many values.   
\end{proof}

We now strengthen $p^3_{**}$ to some $q$
to continuously read $\Gen$ (independently of $\beta$),
again using Fact~\ref{fact:summary}.

So to summarize, we have the following (where 
we start with the 
$\Delta$-system $(M_\beta,p_\beta)_{\beta\in S}$ of Section~\ref{sec:fixit}):
\begin{corollary}\label{cor:qwrqwr}
    There is $\alpha\in \Delta$,  
    $q\in P_\alpha$
    stronger than all $p_\beta\restriction\beta$
    and canonically reading 
    $r_0\le \tilde p$, $\nX$, $\Gen$ and $\npi(\nX)$, such that the following 
    holds for all $\beta\in S$:
    \begin{itemize}
    \item $q\wedge p_\beta$ with the condition\footnote{which is $p_\beta(\beta)=\tilde p$} at index $\beta$ 
    strengthened to $r_0$ is a valid condition, called 
    $p^{++}_\beta$.
        \item $\alpha$, $p^{++}_\beta$ and the names are in $M_\beta$.
        \item $q$ forces in $P_\beta$: $|\nX|=\lambda$,
        $\Gen:\nX\to\npi(\nX)$ is a bijection, and
        if $z\subseteq \nX$ is in $V_\beta$, then
        $p^{++}_\beta/G_\beta\Vdash \npi(z)=^*\Gen''z$.
    \end{itemize}
\end{corollary}

\subsection{Putting everything together}

\begin{corollary}
    (Assuming $\lambda$ is inaccessible and $2^\lambda=\lambda^+$.)
   $P$ forces that every automorphism 
   of $P^\lambda_\lambda$ 
   is somewhere trivial.
\end{corollary}

\begin{proof}
   Assume towards a contradiction that 
   some $p_*$ forces that $\name\phi$ is a nowhere trivial automorphism represented by $\npi$. 
   
   As described in Section~\ref{sec:fixit}
   we find a $\Delta$-system 
   $(M_\beta,p_\beta)_{\beta\in S}$
   with $p_\beta\restriction\beta\le p_*$ for all
   $\beta\in S$,
   and we find 
   $q$, $\nX$, $\Gen$ as in Corollary~\ref{cor:qwrqwr},
   so in particular: $q\le p_\beta\restriction\beta$
   for all $S$; and
   $q$ forces that $|\nX|=\lambda$ and that
   $\Gen:\nX\to\npi(\nX)$ is a bijection.
   
   As $\npi$ is nowhere trivial, $\Gen$ cannot be 
   a generator, i.e., there is some $z\subseteq \nX$
   with $\npi(z)\ne^* \Gen''z$.
   Fix a name for this $z$ and let $q^*\le q$
   canonically read $z$. 

   Pick $\beta\in S$ above $\dom(q^*)$.
   So $q^*\wedge p^{++}_\beta$ is a valid condition,
   which forces that in the final extension $V[G]$ the following holds:
    \begin{itemize}
        \item  $z\subseteq \nX$
   with $\npi(z)\ne^* \Gen''z$, as this is forced by $q^*$.
        \item $z\in V_\beta$, as $q^*$ canonically reads $z$.
        \item So by Corollary~\ref{cor:qwrqwr} and
        as $p^{++}_\beta\in G$, we get
        $\npi(z)=^*\Gen''z$, a contradiction.\qedhere
    \end{itemize}
\end{proof}   

\bibliographystyle{amsalpha}
\bibliography{1224}
\end{document}